\documentclass[review,hidelinks,onefignum,onetabnum]{siamart250211}
\usepackage{lipsum}
\usepackage{amsfonts}
\usepackage{graphicx}
\usepackage{epstopdf}
\ifpdf
\DeclareGraphicsExtensions{.eps,.pdf,.png,.jpg}
\else
\DeclareGraphicsExtensions{.eps}
\fi

\newsiamremark{remark}{Remark}
\newsiamremark{hypothesis}{Hypothesis}
\Crefname{hypothesis}{Hypothesis}{Hypotheses}
\newsiamthm{claim}{Claim}
\newsiamremark{fact}{Fact}
\Crefname{fact}{Fact}{Facts}
\usepackage{amsopn}

\headers{Surface Layers and Water Waves}{T. Askham, T. Goodwill, J. Hoskins, P. Nekrasov, and M. Rachh}

\usepackage{xcolor}
\usepackage{graphicx} 
\usepackage{amsmath,amssymb}
\allowdisplaybreaks[1]  
\usepackage{algpseudocode}
\usepackage[capitalize]{cleveref}
\usepackage{url}
\newcommand{\opnm}[1]{\operatorname{#1}}

\newcommand{\dd}{ {\opnm d} }

\newcommand{\phiinc}{ \phi^{\opnm {inc}} }
\newcommand{\phitot}{ \phi^{\opnm {tot}} }

\newcommand{\pv}{\operatorname{p.\!v.}}

\newcommand{\bbR}{\mathbb{R}}

\newcommand{\bx}{{\mathbf{x}}}
\newcommand{\by}{{\mathbf{y}}}
\newcommand{\bz}{{\mathbf{z}}}
\newcommand{\br}{{\mathbf{r}}}
\newcommand{\bn}{{\mathbf{n}}}
\newcommand{\brprime}{{\mathbf{r}'}}
\newcommand{\brprimeprime}{{\mathbf{r}''}}

\newcommand{\btau}{{\boldsymbol{\tau}}}

\newcommand{\nr}{{\bn_\br}}

\newcommand{\taur}{{\btau_\br}}

\newcommand{\GS}{G_{\opnm{S}}}

\newcommand{\GSfg}{\GS}
\newcommand{\Gphi}{G_\phi}

\newcommand{\threed}{{\opnm{3d}}}

\newcommand{\brthree}{\mathbf{r}_{\opnm{3d}}}
\newcommand{\VS}{\mathcal{V}_{\opnm S}}
\newcommand{\SSi}{\mathcal{S}_{\opnm S}}
\newcommand{\DS}{\mathcal{D}_{\opnm S}}
\newcommand{\KS}{\mathcal{K}_{\opnm S}}
\newcommand{\TS}{\mathcal{T}_{\opnm S}}

\newcommand{\Vphi}{\mathcal{V}_{\phi}}

\newcommand{\Kphi}{\mathcal{K}_\phi}

\newcommand{\Sthreed}{\mathcal{S}_\threed}

\newcommand{\Aop}{\mathcal{A}}
\newcommand{\Bop}{\mathcal{B}}
\newcommand{\Dop}{\mathcal{D}}
\newcommand{\Lop}{\mathcal{L}}
\newcommand{\Mop}{\mathcal{M}}

\newcommand{\Kop}{\mathcal{K}}

\newcommand{\lp}{\left(}
\newcommand{\rp}{\right)}

\newcommand{\GL}{G_\text{L}}
\newcommand{\GB}{G_\text{B}}

\newcommand{\DScgprime}{\DS'}

\newcommand{\Aint}{\mathcal{A}_{\opnm{int}}}
\newcommand{\Aext}{\mathcal{A}_{\opnm{ext}}}

\newcommand{\Lint}{\mathcal{L}_{\opnm{int}}}
\newcommand{\Lext}{\mathcal{L}_{\opnm{ext}}}

\numberwithin{equation}{section}

\title{Surface layers and linearized water waves: a boundary integral equation framework}

\begin{document}
	
	\author{Travis Askham\thanks{Department of Mathematical Sciences, New Jersey Institute of Technology, Newark, New Jersey (\email{askham@njit.edu}).} \and Tristan Goodwill\thanks{Department of Statistics and CCAM, University of Chicago, Chicago, Illinois 
  (\email{tgoodwill@uchicago.edu}).}
  \and Jeremy Hoskins\thanks{Department of Statistics and CCAM, University of Chicago, Chicago, Illinois 
  (\email{jeremyhoskins@uchicago.edu}).}
\and 
Peter Nekrasov\thanks{Committee on Computational and Applied Mathematics, University of Chicago, Chicago, Illinois (\email{pn3@uchicago.edu}).}
\and Manas Rachh\thanks{Department of Mathematics, Indian Institute of Technology Bombay, Mumbai, India 
(\email{mrachh@iitb.ac.in}).}}
	
	\maketitle
	
	\begin{abstract}
		The dynamics of surface waves traveling along the boundary of a liquid medium are changed by the presence of floating plates and membranes, contributing to a number of important phenomena in a wide range of applications. Mathematically, if the fluid is only partly covered by a plate or membrane, the order of derivatives of the surface-boundary conditions jump between regions of the surface. In this work, we consider a general class of problems for infinite depth linearized surface waves in which the plate or membrane has a compact hole or multiple holes. For this class of problems, we describe a general integral equation approach, and for two important examples, the partial membrane and the polynya, we analyze the resulting boundary integral equations. In particular, we show that they are Fredholm second kind and discuss key properties of their solutions. We develop flexible and fast algorithms for discretizing and solving these equations, and demonstrate their robustness and scalability in resolving surface wave phenomena through several numerical examples.
	\end{abstract}
	
	\section{Introduction} \label{sec:intro}
	Linear water wave theories assume that water is an inviscid, incompressible, and 
	irrotational fluid and that the amplitudes of surface waves are small relative
	to their wavelength and the depth of the water. These assumptions allow the fluid motion to be modeled by potential flow, with the kinematic and dynamic boundary conditions prescribed on a fixed surface. The simplicity of the resulting equations
    makes them well-suited to analytical and numerical calculations, and the theory has a long history of success in predicting physical phenomena at many scales, from the coarsening and flexing of biological membranes \cite{crawford1987viscoelastic} to the arrival of transoceanic infragravity waves at Antarctic ice shelves \cite{bromirski2010transoceanic}. 
    
    Building on more analytically-oriented works and contemporary experiments\footnote{In his {\em Tides and Waves}, 
	Airy is remarkably critical of  ``{\em unnecessarily} obscure'' prior work by Laplace and dismissive of 
	``entirely uninteresting'' prior work by Poisson and Cauchy. While he was 
	not alone in these opinions \cite{russell1845report}, their work was influential on the development of wave theories. Airy was also apparently unaware of some related 
	prior work by his contemporaries, including George Green and Philip Kelland. For a review of these historical 
	developments, we refer the interested reader to~\cite{craik2004origins}.}, George Airy published a treatise in 1841 on the motion of the tides and waves in canals~\cite{airy1841tides}, which contained the essential ingredients of modern linear wave theory~\cite{dean1991water}. This theory was later extended by William Thomson 
	in 1871 to include the effects of surface tension \cite{thomson1871xlvi} and by Greenhill in 1886 to include 
	flexural stresses due to elastic plates \cite{greenhill1886wave}. These developments introduced new 
	terms in the dynamic boundary condition at the surface of the fluid in the form of surface differential operators.

	A majority of studies assume that the surface of the fluid has constant properties. However, the effect of spatially varying material properties on the surface is important in many applications. For example, it is well known that the propagation of capillary waves is severely attenuated by the presence of an oil slick, an effect first observed by Reynolds in 1880 \cite{reynolds1880oil}. In the same way, the propagation of sea swell is attenuated by the sea ice pack in the polar regions \cite{wadhams1988,ice6}. 
	
	An interesting class of problems arises when the fluid boundary is divided into regions with different types of surface effects, so that the corresponding boundary conditions in these regions contain differential operators of different orders. 
	This can occur due to the presence of either an opening or occlusion in a membrane or an ice sheet, where the surface dynamics are drastically different from the rest of the surface. Some prototypical examples include polynyas \cite{bennetts2010wave,shi2019interaction}, ice floes \cite{meylan1994response,meylan1996response,meylan2002wave,bennetts2009wave}, open cracks and leads \cite{squire2007ocean,zeng2021flexural}, porous membranes \cite{koley2017oblique}, or partial membranes \cite{karmakar2008gravity,yip2001wave, manam2012mild}. The jumps in the order of the boundary operators generate substantially different dynamics from the continuous problems, including band gaps \cite{chou1998band}, boundary layers \cite{dore1974edge}, and superlensing \cite{hu2004superlensing}. Similar problems also arise in the modeling of acoustic boundary layers in loudspeaker components \cite{berggren2018acoustic}.
	
	The aim of the present work is to demonstrate that, particularly for ``deep water'' models in which the fluid occupies a half-space,
	there is an effective, simple, and unified approach to the numerical solution of these mixed-order boundary condition problems 
	based on integral equation formulations. More concretely, we consider linear water waves in the time harmonic setting,
	in which the spatial dependence of the fluid flow is written in terms of a velocity potential $\phi : H^- \to \mathbb{C}$, 
	where $H^-$ is defined as the lower half-space $H^- := \{(x,y,z) \in \mathbb{R}^3 \, | \, z < 0 \}$ whose boundary, $D$, we identify with $\mathbb{R}^2$. We let $\Omega$ be a bounded domain on the fluid surface $D$ and let $D \setminus \overline{\Omega}$ denote 
	the exterior of this region, also on the surface. The boundary value problem takes the form:
	\begin{equation}
		\begin{cases}
			\begin{aligned}
				(\Delta + \partial_z^2) \phi &= 0 \, ,  \, && \text{ in } H^- \, , \\
				\Aext[\phi, \partial_z \phi ] &= 0 \, , \, && \text{ in } D \setminus \overline{\Omega} \, ,  \\
				\Aint [\phi, \partial_z \phi ] &= f \, , \, && \text{ in }  \Omega \, ,  \\
				\mathcal{B}[\phi, \partial_z \phi] &= g \, , \, && \text{ on } \partial \Omega \, , 
			\end{aligned}
		\end{cases} \label{eq:bvps}
	\end{equation}
	where $\Delta := \partial_x^2 + \partial_y^2$ is the Laplace operator in two dimensions. See~\Cref{fig:setup} for an illustration. The first equation corresponds to an incompressibility condition in the fluid bulk, the second and third conditions correspond to dynamic boundary conditions on the exterior and interior regions of the fluid surface, respectively, and the last equation corresponds to an in-plane boundary condition which may be vector-valued, depending on the order of the operators $\Aext$ and $\Aint$. 
	
	The inputs to $\Aext$ and $\Aint$ should be understood as the traces of $\phi$ and $\partial_z\phi$ on $D$.
	Note that the trace of $\partial_z\phi$ on $D$ is the vertical velocity of the surface of the fluid, which is related to the vertical displacement through a kinematic condition, while $\phi$ is related to the pressure of the fluid through Bernoulli's principle. These two quantities are connected through $\Aext$ and $\Aint$, which have the following general form: 
	\begin{align}
		\Aop[\phi, \partial_z \phi ] = p(-\Delta) \partial_z \phi - \phi \; ,  \label{eq:Aphi}
	\end{align}
	where $p$ is a polynomial. Quadratic terms in $p$ account for flexural effects, linear terms for
	elastic and surface tension effects, and constant terms for inertial and gravitational effects. 
	In this work, we assume that the polynomial degree is greater for the $\Aext$ operator than it is
	for the $\Aint$ operator, though the case where this condition is reversed can be treated with similar tools.
	
	\begin{figure}
		\centering
		\includegraphics[width=0.6\linewidth]{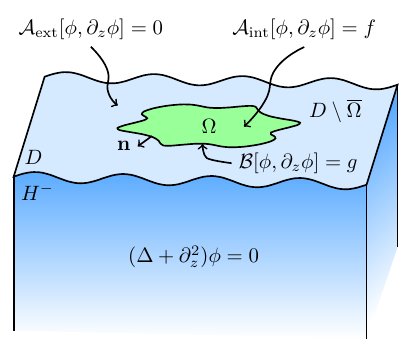}
		\caption{General setup of the problem.}
		\label{fig:setup}
	\end{figure}
	
	Problems of a similar flavor have been considered in a number of different contexts. Many studies have examined the effects of continuous changes in the surface properties such as the thickness of an ice cover \cite{porter2004approximations,williams2004oblique,bennetts2007multi,nekrasov2023ocean,askham2025integral} or variations in surface tension \cite{chou1994surface,shen2017marangoni}. These papers typically look at simplified geometries or idealized settings using semi-analytical approaches. It is also possible to have a discontinuous jump in the coefficients on the surface, where it is typically necessary to prescribe additional transmission or continuity conditions. Some early work looked at discontinuities in the surface tension for simple geometries, such as on an infinite flat boundary with the help of a Fourier-type integral representation \cite{gou1993capillary} or the Wiener-Hopf method \cite{dore1974edge}, or on circular domains with the help of truncated Bessel series \cite{chou1995capillary} or the first order Born approximation \cite{chou1994surface}.
	
	Perhaps the most closely related method is that of \cite{bennetts2010wave}, which focuses on the polynya problem in the finite depth case. In their approach, the solutions in the ice-covered and ice-free regions are expanded in terms of vertical modes. Green's identities are then employed to reduce the problem to a system of integro-differential equations along the boundary of these two regions, which are solved using a Galerkin method. While this method offers the flexibility needed to incorporate draft and finite depth, it involves discretizing the fluid where no significant dynamics occur. In contrast, our approach seeks to reduce the problem to an integral equation defined solely on the surface of the polynya or, more generally, the object responsible for scattering. We remark that it is similar in spirit to the \emph{surface wave preconditioners} developed in \cite{kleingordon_waveguide, dirac_waveguide} for singular waveguides and topological insulators.
	
	The numerical approach advocated here is characterized by the use of a nested integral representation of the 
	solution. The original three dimensional problem in $D$ is first reduced to a two-dimensional integro-differential 
	equation on the surface $ D$ using a \emph{single layer} integral representation for $\phi$. The resulting integro-differential
	equation, defined on the infinite interface $ D$, is further reduced to the compact set $\Omega$  
	using the Green's function for the integro-differential analog of the $\Aext$ operator. We show that it
	is relatively straightforward to select this second representation so that the resulting equation is a Fredholm second
	kind integral equation system defined on $L^2(\Omega)\times L^2(\partial \Omega)^m$, where $m$ denotes the number of boundary conditions imposed by $\mathcal{B}$. More details of this overall framework are provided in \Cref{sec:reduction2d} and the
	Green's functions and their properties are derived in \Cref{sec:green}. 
	We derive appropriate representations for two common applications and discuss the invertibility of
	the resulting integral equation systems in \Cref{sec:reductioninteq}. 
	
	Once suitable integral representations are defined for these problems, the corresponding integral equations can be 
	discretized using standard tools. We describe a discretization scheme in \Cref{sec:discretization},
	based on a high-order triangulation of $\Omega$ and panelization of $\partial \Omega$. To treat a mild singularity
    which develops in $\partial_z\phi$ near the interface, we apply adaptive refinement near the boundary of $\Omega$. The resulting linear systems are dense and require acceleration to scale to larger problems. While the system matrices have similar rank structure
	to the matrices resulting from integral equation representations of PDE solutions, namely that 
	submatrices corresponding to well-separated sets of source and target points are low rank, many of the common
	fast algorithm techniques do not apply directly to these systems. We describe a simple and efficient scheme for 
	accelerating matrix-vector multiplication for these systems based on pre-corrected FFTs in \Cref{sec:algorithm}.
	The scheme uses \emph{proxy annuli} to efficiently compress interactions of points which are physically well-separated. Similar ideas have been employed for compressing kernels arising in Gaussian process regression, see~\cite{minden2017fast}, for example. Finally, in \Cref{sec:examples} we present several numerical examples illustrating the efficacy of the proposed framework. We conclude with \Cref{sec:discussion}, in which we outline directions for future work.
	
	\section{Overview of the method} \label{sec:reduction2d}
	
	We represent the velocity potential $\phi$ in the fluid domain using the three-dimensional Laplace single layer $\Sthreed$ applied to an unknown density $\sigma$ defined on the surface (see also~\cite{chou1994surface,fox1999green,de2018capillary,oza2023theoretical,askham2025integral}):
	\begin{equation}
		\phi(\brthree) = \Sthreed[\sigma] (\brthree) :=  \pv \int_{\mathbb{R}^2} \frac{1}{4\pi | \brthree - \brthree'|} \sigma(\brprime) \, \dd A(\br') \, , \label{eq:Sdef}
	\end{equation}
	where $\brthree = [x, y, z]^T ,   \brthree' = [x', y', 0]^T , \, \brprime = [ x', y' ]^T $, and the principal value integral ($\pv$) is defined as
	\begin{equation*}
		\pv \int_{\mathbb{R}^2} f(\brprime) \, \dd A(\brprime) = \lim_{R \to \infty} \int_{|\br| \leq R} f(\brprime) \, \dd A(\brprime) \, .
	\end{equation*}
	The principal value definition can accommodate slowly decaying but oscillatory densities, like the ones we will obtain for certain parameters. Applying the ansatz~\cref{eq:Sdef} ensures that $\phi$ satisfies Laplace's equation and decays in the fluid region $H^-.$ All that remains is to choose the density $\sigma$ to satisfy the boundary conditions on $D.$

	The standard jump relations for the Laplace single layer potential \cite{kress1999linear} imply that the limit of $\Sthreed[\sigma]$ can be taken continuously to the boundary for any sufficiently smooth density $\sigma$ with mild conditions on the decay and oscillations at infinity, i.e. letting $\brthree = [x,y,0]^T , \, \br = [x,y]^T , \, \mathbf{z} = [0,0,1]^T$, and $h > 0$ we have
	\begin{equation*}
		\lim_{\substack{h \to 0}} \mathcal{S}_{\threed}[\sigma](\brthree - h \mathbf{z}) = \mathcal{S}_{\threed}[\sigma](\brthree) \, ,
	\end{equation*}
	while the normal derivative is given by
	\begin{equation*}
		\lim_{\substack{h \to 0}} \partial_z \mathcal{S}_{\threed}[\sigma](\brthree - h \mathbf{z}) =
		\frac{1}{2} \sigma(\br) + \partial_z \mathcal{S}_{\threed}[\sigma](\brthree) = \frac{1}{2} \sigma(\br) \; .
	\end{equation*}
	Substituting the single layer potential representation for $\phi$ into \eqref{eq:Aphi} and applying these jump relations allows us to write the boundary operator \eqref{eq:Aphi} in terms of the density $\sigma(\br)$. We call this new integro-differential operator $\Lop[\sigma]$:
	\begin{equation} 
		\mathcal{L}[\sigma](\br) := \mathcal{A}[\Sthreed[\sigma],\sigma/2](\br) = \frac{1}{2} p(-\Delta) \sigma(\br) - \int_{\mathbb{R}^2} \frac{1}{4\pi | \br - \brprime|} \sigma(\brprime) \, \dd A(\brprime) \; . \label{eq:Lsigma}
	\end{equation}
	Similarly, substituting our ansatz for $\phi$ into the in-plane boundary operator $\Bop$ 
    in \eqref{eq:bvps} yields another operator $\Mop[\sigma]:= \Bop[\Sthreed[\sigma],\sigma/2]$. This converts the three-dimensional boundary value problem problem given by \eqref{eq:bvps} into a two-dimensional surface problem of the form:
	\begin{equation}
		\begin{cases}
			\begin{aligned}
				\Lext[\sigma] &= 0 \, , \, && \text{ in } D \, \backslash \, \overline{\Omega} \, ,  \\
				\Lint [\sigma] &= f \, , \, && \text{ in }  \Omega \, ,  \\
				\Mop[\sigma] &= g \, , \, && \text{ on } \partial \Omega \, , 
			\end{aligned}
		\end{cases} \label{eq:reducedbvp}
	\end{equation}
	where $\Lext$ and $\Lint$ are integro-differential operators of the general form \eqref{eq:Lsigma}, with the order of the highest derivative of $\Lext$ being strictly greater than the order of the highest derivative in $\Lint$, and $\Mop$ a possibly vector-valued operator. 
    
    In order to solve this problem, we introduce the following new ansatz involving both a \emph{surface-volume} density $\mu \in L^2(\Omega)$ and a \emph{surface-boundary} density $\eta \in (L^2(\partial \Omega))^m$:
	\begin{equation}\sigma(\br) = \int_\Omega V(\br,\brprime) \mu(\brprime) \, \dd A(\brprime) + \int_{\partial \Omega} B (\br,\brprime) \cdot \eta (\brprime) \, \dd s(\brprime) \, , \label{eq:ansatz}
	\end{equation}
	where $V(\br,\brprime) : \mathbb{R}^2 \times \mathbb{R}^2 \to \mathbb{C} $, $B(\br,\brprime) : \mathbb{R}^2 \times \mathbb{R}^2 \to \mathbb{C}^m $, and $m$ is the number of vector components of $\Mop$.
	If the kernels $V(\br,\brprime)$ and $B(\br,\brprime)$ are constructed out of the Green's function corresponding to the exterior operator, then  the ansatz \eqref{eq:ansatz} will satisfy the homogeneous equation in the exterior. We defer a detailed proof of this to \Cref{sec:green}. Thus, the only constraints on the densities $\mu$ and $\eta$ result from the failure of the ansatz \eqref{eq:ansatz} to satisfy the interior equation $\Lint[\sigma] = f$ in $\Omega$ and the boundary condition $\Mop[\sigma] = g$ on $\partial \Omega$. Substituting our ansatz into these two equations results in a system of coupled surface-volume and surface-boundary integral equations defined on $L^2(\Omega) \times (L^2(\partial\Omega))^m \to L^2(\Omega) \times (L^2(\partial\Omega))^m$: 
	\begin{align}
		T_0 \mu(\br) + \int_\Omega  K_{00} (\br,\brprime) \mu(\brprime) \, \dd A(\brprime) + \sum_{j=1}^m \int_{\partial \Omega}  K_{0j}(\br,\brprime)   \eta_j (\brprime) \, \dd s(\brprime) &= 
		f(\br) \, , \label{eq:vie} \\ 
		T_{i} \eta_i(\br) + \int_\Omega K_{i0}(\br,\brprime) \mu(\brprime) \, \dd A(\brprime) + \sum_{j=1}^m \int_{\partial \Omega} K_{ij}(\br,\brprime)  \eta_j (\brprime) \, \dd s(\brprime) &= g_i(\br)  \, , \label{eq:bie}
	\end{align}
	for $i = 1, \dots , m$. There are two different ways in which the identity term $T$ appears. In the surface-volume equation \eqref{eq:vie} the term $T_0 \mu$ appears because the interior operator $\Lint$ applied to the kernel $V(\br,\brprime)$ contains a delta function which is handled explicitly. In the surface-boundary equations \eqref{eq:bie} the identity terms $T_i \eta_i$ appear when taking the exterior traces of the layer potentials at the boundary. In other words, we define $T_i \eta_i$ as the difference between the off-surface limit of the integral operator and its on-surface value, i.e. for $\br \in D \setminus \overline{\Omega}$ and $\br_0 \in \partial \Omega$ we have:
	\begin{equation*}
		T_i\eta_i(\br_0) := \lim_{\br\to \br_0} \int_{\partial \Omega} K_{ii}(\br,\brprime)\eta_i(\brprime) \, \dd A(\brprime) -  \pv \int_{\partial \Omega} K_{ii}(\br_0,\brprime)\eta_i(\brprime) \, \dd A(\brprime)
	\end{equation*}
	where the limit is taken in the normal direction. Then, the kernels in the integral equations \eqref{eq:vie}-\eqref{eq:bie} are defined as:
	\begin{align*}
		K_{00}(\br,\brprime) &= \Lint [V(\ \cdot \ ,\brprime)](\br) - T_0 \delta(\br,\brprime) \, , \\
		K_{0j}(\br,\brprime) &= \Lint  [B_j(\ \cdot \ ,\brprime) ] (\br) \, , \\
		K_{i0}(\br,\brprime) &= \Mop_i V(\br,\brprime) \, , \\
		K_{ij}(\br,\brprime) &= \Mop_i B_j(\br,\brprime) \, .
	\end{align*}
    
	In this paper, through two representative examples, we show how to construct the kernels $V(\br,\brprime)$ and $B(\br,\brprime)$ so that equations \eqref{eq:vie}-\eqref{eq:bie} form a system of Fredholm integral equations of the second kind. In general, a Fredholm integral equation of the second kind is an integral equation of the form:
	\begin{equation*}
		(I + K) u = f
	\end{equation*}
	for $u, f \in X$, where $X$ is a Banach or Hilbert space, $I$ is the identity operator, and $K: X \to X$ is a compact integral operator acting on $X$. Such equations have a number of numerical and analytical advantages, the most notable being that the condition number of the corresponding discrete linear system remains bounded under suitable refinement of the discretization \cite{kress1999linear}. 
	
	The solution $\phi$ to the boundary value problem \eqref{eq:bvps} can be retrieved from the ansatz \eqref{eq:ansatz} through the following formula:
	\begin{multline}
		\phi(\brthree) = \int_{\mathbb{R}^2} \frac{1}{4\pi |\brthree-\brthree'| }  \int_\Omega V(\brprime,\brprimeprime) \mu(\brprimeprime) \, \dd A(\brprimeprime) \, \dd A(\brprime) \\
		 + \int_{\mathbb{R}^2} \frac{1}{4\pi |\brthree-\brthree'| } \int_{\partial \Omega} B (\brprime,\brprimeprime) \cdot \eta (\brprimeprime) \, \dd s(\brprimeprime) \, \dd A(\brprime) \, , \label{eq:phiformula}
	\end{multline}
	where $\brthree' = [x',y',0]^T$ and $\brprime = [x',y']^T$. Later, we show that the integrals above can be interchanged and that simple analytical formulas are available for $\Sthreed [ V( \, \cdot \, ,\brprimeprime)] (\br) $ and $\Sthreed [ B( \, \cdot \, ,\brprimeprime)] (\br)$ when $\br'',\br\in D$. Using these formulas, $\phi$ can be computed on surface by evaluating integrals over finite regions. 
	
	\subsection{Notation and assumptions}
	
	In the remainder of the paper, we assume that $\Omega$ is a bounded, open domain in $\bbR^2$, whose boundary, $\partial \Omega$, 
	is a smooth, regular curve. The in-plane outward-pointing normal, positively-oriented tangent, and 
	signed curvature at a point $\br \in \partial \Omega$ are denoted by $\bn(\br)$, $\btau(\br)$, and $\kappa(\br)$, 
	respectively, and the dependence on $\br$ is dropped when it is clear from context. For a function of two variables, 
    $K(\br,\br')$, normal derivatives are denoted by 
    $$ \partial_{\bn} K(\br,\br') = \bn(\br)\cdot \nabla_{{\bf w}} K({\bf w},\br') |_{{\bf w}=\br} \quad \textrm{ and } \quad
    \partial_{\bn'} K(\br,\br') = \bn(\br') \cdot \nabla_{\bf w} K(\br,{\bf w}) |_{{\bf w}=\br'} \; ,$$
	where $\br,\br' \in \partial \Omega$, as appropriate. A similar notation is used for tangential derivatives. Regularity results below are described using the standard Sobolev spaces, denoted $H^s$. By $H^s_{\rm loc}(\bbR^2)$, we 
	mean the space of functions $f$ such that $\varphi f \in H^s(\bbR^2)$ for any compactly-supported, 
	smooth $\varphi$.
	
	For several quantities in the analysis, the restriction of a function
	to the boundary $\partial \Omega$ must be understood in terms of boundary traces.
	Let $\tilde{\Omega}\subset D$ be an open domain that contains $\overline{\Omega}$.
	For a function $f$ defined on $\tilde{\Omega}$, $\gamma_0^-f$ denotes the trace of $f \upharpoonright_\Omega$ on $\partial \Omega$, 
	which is well-defined provided the restriction $f \upharpoonright_\Omega$ has sufficient regularity. 
	Likewise, $\gamma_0^+f$ denotes the trace of $f\upharpoonright_{\tilde{\Omega}\setminus 
		\overline{\Omega}}$ on $\partial \Omega$. The ``jump'' in $f$ across the boundary is denoted
	by $[[f]]:= \gamma_0^+f-\gamma_0^{-}f$. The notations $\gamma_j^+ f$ and $\gamma_j^- f$ 
	refer to the exterior and interior traces of the $j$th normal derivative of $f$, respectively.
        In many cases a layer potential with an $L^2$ density has well-defined 
        interior and exterior limits from the normal direction to the boundary (in the $L^2$ sense), even though the 
        restriction of the layer potential to $\Omega$ or $\tilde{\Omega}\setminus \overline{\Omega}$ 
        might not have sufficient regularity
        for the trace theorem. For ease of exposition, and with a slight abuse of notation, we will still use the trace notation in these instances.
	
	\section{Green's functions and regularity of integral operators} 
	\label{sec:green}
	Let the operator $\mathcal{L}$ be in the form of \eqref{eq:Lsigma} and let 
	$d_p$ be the degree of $p$. The fundamental solution $\GS$ satisfying the equation
	\begin{equation}\label{eq:f_gf}
		\mathcal{L}[\GS(\ \cdot \ ,\br')](\br) = \delta(\br-\br') \; ,
	\end{equation}
	together with suitable radiation conditions at infinity, can be explicitly constructed using standard Fourier methods. In this section, we give a brief sketch of this derivation, along with several analytic properties of $\GS$ and integral
    operators derived from $\GS$, which will be used in later sections. 
	
	We begin by noting that the translational invariance of $\mathcal{L}$ immediately implies that $\GS(\br,\br') = \GS(\br-\br')$ and so, without loss of generality we may assume that $\br' = 0.$ Taking a Fourier transform of \eqref{eq:f_gf} yields
	\begin{align}\label{eqn:poly}
		\left[ \frac{1}{2} p(\xi^2) -\frac{1}{2|\xi|}\right] \tilde{G}(\xi) = 1 \; ,
	\end{align}
	where $\tilde{G}$ is the Fourier transform of $\GS(\br).$ The bracketed quantity in the above equation can be written in the form $P(|\xi|) |\xi|^{-1}$ where $P(\xi)=(\xi p(\xi^2) - 1)/2$ is a polynomial of degree $d_P =2d_p+1.$ Let $\rho_1,
    \cdots,\rho_{d_P}$ denote the roots of $P.$ Then, performing a partial fraction decomposition, we have
	\begin{align*}
		\tilde{G}(\xi) = |\xi| \sum_{j=1}^{d_P} \frac{c_j}{|\xi|-\rho_j}, \quad \xi \in \mathbb{R}^2,\quad |\xi| \neq 0,\rho_1,\cdots,\rho_{d_P} \; ,
	\end{align*}
	where $c_j = \lim_{z\to \rho_j} (z-\rho_j)/P(z).$ Here we assume for simplicity that all of the roots of $P$ are simple. Similar results can be obtained for roots of higher multiplicity. 
	
	As a consequence of Cauchy's theorem, the coefficients $c_j$ satisfy the moment conditions: $\sum_{j=1}^{d_P} c_j \rho_j^\ell =0$ for $\ell = 0,\cdots,d_P-2$ and $\sum_{j=1}^{d_P} c_j \rho_j^{d_P-1} = \frac{1}{a_{d_P}}$, where $a_{d_P}$ is the leading order coefficient of $P(z)$. Moreover, since the coefficient of $z^{d_P-1}$ in $P(z)$ is zero, the above moment conditions together with the identities $P(\rho_j) =0,$ $j=1,\cdots,d_P$ imply that the $d_P$th moment also vanishes. 
	
	In particular, it follows that for $d_P \ge 3$
	$$\tilde{G}(\xi) =  \sum_{j=1}^{d_P} \frac{c_j\rho_j}{|\xi|-\rho_j},$$
	again excluding the set $|\xi| = \rho_1,\cdots,\rho_{d_P}.$ 
    For $d_P\le 2$ an extra constant term is present in the above expression leading to a delta function in the final expression for $\GS.$  
	
	For $\rho_j$ not on the non-negative real axis, we may use the identity \cite{askham2025integral}
	\begin{align}
		\mathcal{F}^{-1}[1/(|\xi|-\rho_j)](\br) = \frac{1}{2\pi |\br|}+\frac{\rho_j}{4}{\rm \bf K}_0(-\rho_j |\br|) \; , \label{eq:struveid}
	\end{align}
	where $\mathcal{F}^{-1}$ is the inverse Fourier transform and ${\rm \bf K}_0$ is the Struve function \cite{NIST:DLMF}. For $\rho_j \in \mathbb{R}^+,$ the limit 
	approaching $\rho_j$ from
	the upper half of the complex plane is
	\begin{equation*}
		 \lim_{\epsilon \to 0^+} \mathcal{F}^{-1}[1/(|\xi|-\rho_j - i \epsilon)](\br) = \frac{1}{2\pi |\br|} -\frac{\rho_j}{4}{\rm \bf K}_0(\rho_j |\br|)+\frac{i \rho_j}{2} H_0^{(1)}(\rho_j |\br|) \, ,
	\end{equation*}
	giving an {\it outgoing} contribution to the fundamental solution, where $H_0^{(1)}$ denotes the zeroth order 
	Hankel function of the first kind. 
	
	\begin{remark}
		In most applications there is at most one positive root and the outgoing solution can be 
		obtained by applying a limiting absorption principle. For multiple positive roots one would expect a limiting absorption principle to also apply, though the choice of branch cuts may depend on the problem setting.
	\end{remark}
	
	Using these formulas, as well as the moment conditions, we have that for $d_P \ge 3,$
	\begin{align}\label{eqn:general_G}
		\GS(\br) = \frac{1}{4}\sum_{\rho_j \notin\mathbb{R}^+}c_j \rho_j^2{\rm \bf K}_0(-\rho_j|\br|) + \sum_{\rho_j \in \mathbb{R}^+} c_j \rho_j^2\left[ -\frac{1}{4}{\rm \bf K}_0(\rho_j |\br|) + \frac{i}{2} H_0^{(1)}(\rho_j|\br|)\right]
	\end{align}
	is a solution of \cref{eq:f_gf}.
	For $d_P <3,$ additional terms involving delta functions and multiples of $1/|\br|$ will also appear.
	
	Similarly, if we define
	\begin{equation}
		\label{eq:Gphiid}
		\Gphi (\br,\br') :=  \pv \int_{\mathbb{R}^2}
		\frac{1}{ 4\pi \left |\br-\br'' \right|} \GS(\br'',\br') \,
		\dd A(\br'') \; ,
	\end{equation}
	then $\Gphi$ is also translationally invariant so that $\Gphi(\br,\br') = \Gphi(\br-\br')$. Moreover,
	\begin{align}
    \label{eqn:general_Gphi}
		\Gphi(\br) = \frac{1}{8}\sum_{\rho_j \notin\mathbb{R}^+}c_j \rho_j{\bf K}_0(-\rho_j|\br|) +  \sum_{\rho_j \in \mathbb{R}^+} \frac{c_j}{2} \rho_j\left[ -\frac{1}{4}{\bf K}_0(\rho_j |\br|) + \frac{i}{2}  H_0^{(1)}(\rho_j|\br|)\right]. 
	\end{align}
	
	Note that for $p$ linear the corresponding Green's functions have appeared in previous studies on capillary surfers~\cite{de2018capillary,oza2023theoretical} and surface active substances \cite{chou1994surface}. For quadratic $p$ the Green's functions have appeared in \cite{fox1999green,askham2025integral}.
	
	In many problems of physical interest, there is only one propagating frequency in the Green's function.
	Furthermore, if there is dissipation added to the surface, then there are no slowly decaying modes.
	This is summarized in the following proposition.
	\begin{proposition} \label{prop:sommerfeld}
		Suppose that $p(z) = a_{d_p} z^{d_p} + \cdots + a_1 z + a_0$ with $d_p\geq 1$. 
        If $a_1,\ldots,a_{d_p-1} \geq 0$, $a_{d_p}>0$, and $a_0$ is real, then $P(z) = (zp(z^2)-1)/2$ has 
		precisely one positive real root, which we
		take to be $\rho_1$. In this case, $\GS(\br)$ and $\Gphi(\br)$ radiate like the Hankel function 
		$H^{(1)}_0(\rho_1 |\br|)$, i.e. 
		\begin{equation} \frac{\br}{|\br|} \cdot \nabla_\br\GS(\br) - i \rho_1\GS(\br) = o \left (\frac{1}{\sqrt{|\br|}} \right ) 
			\quad \textrm{and} \quad \frac{\br}{|\br|} \cdot \nabla_\br\Gphi(\br) - i \rho_1\Gphi(\br) = o \left (\frac{1}{\sqrt{|\br|}} \right ) \; ,     
			\label{eq:sommerfeld}
		\end{equation}
		and likewise for their derivatives, as $|\br|\to\infty \; .$
		
		In the case that the coefficients $a_{1},\ldots,a_{d_p}$ are real and 
		$a_0$ has non-zero imaginary part, none of the roots of $P$ are real. Then,
		\begin{equation}
			\GS(\br) = O \left (\frac{1}{|\br|^3} \right )  \quad \textrm{ and } \quad
			\Gphi(\br) = O \left (\frac{1}{|\br|^3} \right ) \; , \label{eq:decaydissipative}
		\end{equation}
		and likewise for their derivatives. 
		More generally, these decay rates hold whenever none of the roots of $P$ are
		real and positive.
	\end{proposition}
    \begin{proof}
         The facts about the roots follow from elementary properties of 
	polynomials. The decay conditions follow by applying the large $|\br|$ asymptotics of Struve and Hankel functions~\cite{NIST:DLMF} and the moment conditions to \cref{eqn:general_G} and 
    \cref{eqn:general_Gphi}.    
    \end{proof}
	
	The slow decay of the Green's functions in the real coefficient case adds additional technical difficulties
	to the analysis. For ease of exposition, in several of our results we restrict our attention to the \emph{dissipative regime} defined below.
	
	\begin{definition}[Dissipative regime]
		The coefficients in the polynomial $p$ are in the {\em dissipative regime}
		if $P$ has no real, positive roots.
	\end{definition}
	
	The small $|\br|$ asymptotics of the Green's functions can be obtained from Fourier analysis and \cref{eqn:poly}. 
	Details of the formulas can be deduced from the small $|\br|$ asymptotics of Struve and Hankel functions and the moment 
	conditions. We provide these in the following lemma; its proof is straightforward and so we omit it.
	\begin{lemma} \label{lem:smallr}
		Suppose $d_P \ge 3.$ Then $\GS$ admits a small $\br$ expansion of the form
		$$\GS(\br) = A_{\rm S}(|\br|^2) + |\br|^{d_P} B_{\rm S}(|\br|^2) +  |\br|^{d_P-3}\log|\br| C_{\rm S}(|\br|^2) \; $$
		and $\Gphi$ has an expansion of the form
		$$\Gphi(\br) = A_\phi(|\br|^2) + |\br|^{d_P-2} B_\phi(|\br|^2) +  |\br|^{d_P+1}\log|\br| C_\phi(|\br|^2) \; ,$$
		where the $A_{\rm S/\phi}$, $B_{\rm S/\phi}$, and $C_{\rm S/\phi}$ functions are infinitely differentiable. 
		
		In particular, when $p(z) = a_1 z + a_0 $, 
		$\frac{a_1}{2} \GS + G_{\rm L} \in H^2_{\rm loc}(\mathbb{R}^2),$ where $G_{\rm L}(\br) = -\frac{1}{2\pi} \log | \br|$ is the standard Laplace Green's function. Similarly, for $p(z) = a_2 z^2+ a_1 z+ a_0$, $\frac{a_2}{2} \GS- G_{\rm B} \in H^4_{\rm loc}(\mathbb{R}^2),$ where $G_{\rm B}(\br) = \frac{1}{8\pi} |\br|^2 \log|\br|$ is the standard biharmonic Green's function.
	\end{lemma}
	
	Consider the volume integral operators 
	$$  \mathcal{V}_{\rm S}[\mu](\br) := \int_{\Omega} \GS(\br,\br') \mu(\br') \, {\rm d}A(\br')\quad \textrm{and} 
	\quad \mathcal{V}_{\rm \phi}[\mu](\br) := \int_{\Omega} \Gphi(\br,\br') \mu(\br') \, {\rm d}A(\br') \; .$$
        We summarize some smoothing properties of these operators in the following lemma. Because the Green's functions
        are not PDE Green's functions, higher regularity is described for compactly contained subsets. The
        arguments for these properties are standard and omitted. 
	\begin{lemma}
		\label{lem:genregvolume}
		Suppose that $d_P\geq 3$, let $A \Subset \Omega$ be an open set compactly contained in $\Omega$, and let
        $m \in \mathbb{N}_0$ be given. Then, $\VS:L^2(\Omega)\to H^{d_P-1}_{\rm loc}(\bbR^2)$, $\Vphi:L^2(\Omega)\to H^{d_P}_{\rm loc}(\bbR^2)$,
		$\VS:H^m(\Omega)\to H^{m+d_P-1}(A)$, and $\Vphi:H^m(\Omega)\to H^{m+d_P}(A)$. 
        The ``loc'' may be dropped in the dissipative regime. 
	\end{lemma}
\noindent The regularity of the corresponding boundary-to-volume operators can also be established.
    \begin{lemma}
    \label{lem:genregbdry}
        		Let 
		$$ K_{\rm S}(\br,\br') = \partial_{\bn'}^{\ell} \partial_{\btau'}^m 
		\GS(\br,\br') \quad \textrm{and} \quad   K_{\phi}(\br,\br') = \partial_{\bn'}^{\ell} \partial_{\btau'}^m 
		\Gphi(\br,\br') \; , $$ 
		where $\ell+m \leq d_P-2$. Let $\KS$ and $\Kphi$ be the corresponding integral operators
		$$ \KS[\eta](\br) = \int_{\partial \Omega} K_{\rm S}(\br,\br') \eta(\br') \, \dd s(\br') \quad \textrm{and} \quad
		\Kphi[\eta](\br) = \int_{\partial \Omega} K_{\phi}(\br,\br') \eta(\br') \, \dd s(\br') \; .$$
		Then,
		$\KS:H^s(\partial \Omega)\to H^{s+d_P-3/2-\ell-m}(\Omega)$ for $0\leq s \leq 3/2$ and 
        $\Kphi: L^2(\partial \Omega)\to H^{d_P-1-\ell-m}(\Omega)$.
    \end{lemma}
	\begin{proof}
	Recall that $d_P \geq 3$ is odd. For $\GS$, the dominant term
        in the expansion provided by \Cref{lem:smallr} is $|\br|^{d_P-3}\log |\br|$ and up to 
        $d_P$ derivatives of the other terms are bounded. The dominant term is the Green's function 
        of an elliptic PDE of order $d_P-1$; the pseudo-differential theory, e.g. Theorem 8.5.8
        in \cite{hsiao2008boundary}, then implies that the contribution of the dominant term 
        gains $d_P-3/2$ derivatives for boundary-to-volume. For $\Gphi$, the dominant term is 
        $|\br|^{d_P-2}$, which is not a PDE kernel. Observe that taking $d_P-1$ derivatives of the 
        dominant term results in singularities of the form $1/|\br|$, which give bounded operators from
        $L^2(\partial \Omega)$ to $L^2(\Omega)$.
	\end{proof}
    
	Given these regularity results and asymptotics, it is relatively straightforward to derive the 
	compactness and jump properties of the integral operators we define for specific boundary value
	problems in the next section. We will make frequent use of some standard results from Sobolev
	space theory and integral equation theory that we collect below. 
	
	For the volume-to-boundary operators, regularity of the trace can be established by combining
    \Cref{lem:genregvolume} and the trace theorem; see, e.g., \cite[Theorem 4.2.1]{hsiao2008boundary}. The version below 
	is a straightforward consequence of the usual statement.
	\begin{theorem}[Trace theorem]
		We have $\gamma_0^-:H^1(\Omega)\to H^{1/2}(\partial \Omega)$ and 
		$\gamma_0^+:H^1_{\rm loc}(\bbR^2\setminus \overline{\Omega})\to H^{1/2}(\partial \Omega)$.
		Moreover, if $f\in H^1_{\rm loc}(\bbR^2)$, then $\gamma_0^+f = \gamma_0^-f$.
	\end{theorem}
	
	\noindent Compactness can then be inferred 
	by applying Rellich's lemma~\cite[Theorem 4.1.6]{hsiao2008boundary}.
	\begin{lemma}[Rellich's lemma]
		Let $s > t$. The embeddings $H^s(\Omega) \hookrightarrow H^t(\Omega)$ 
		and $H^s(\partial \Omega) \hookrightarrow H^t(\partial \Omega)$ are 
		compact.
	\end{lemma}
	
	As observed above, the $\GS$ kernel is a relatively smooth perturbation of a (poly)-harmonic 
	kernel for $d_P\geq 3$. To establish the compactness of the remainder, we will apply 
	the following general result for integral operators with weakly singular kernels, which
	is adapted from~\cite[Theorem 2.22]{kress1999linear}.
	\begin{lemma}
		\label{lem:gencompact}
		Suppose that $K(\br,\br')$ is continuous except on the diagonal, 
		$\br=\br',$ and that $|K(\br,\br')| \leq M|\br-\br'|^\nu$ for $|\br-\br'|\leq 1$ and 
		some constants $M\geq 0$ and $\nu \in (-1,0]$. Then, the integral operator with kernel $K$ is compact on $L^2(\partial \Omega)$. 
	\end{lemma}
	
	Before proceeding, we show that volume potentials and boundary layer potentials defined using the Green's
	function are indeed solutions of the homogeneous equation outside of the support of their densities.
	For PDE Green's functions this is a trivial step, but here the nonlocal terms in~\cref{eq:Lsigma} require
	some care. 
	
	\begin{lemma}\label{lem:sgsgphi}
		Suppose that $d_P\geq 3$ and that the coefficients are in the dissipative regime. Let 
		$\mu \in L^2(\Omega)$ and $\eta \in L^2(\partial \Omega)$ be given. Suppose that 
		$K_{\rm S}(\br,\br') = \partial_{\bn'}^{j} \partial_{\btau'}^k \GS(\br,\br')$ for $j+k \leq d_P-2$
		and let $\KS$ be the corresponding integral operator 
		$$ \KS[\eta](\br) = \int_{\partial \Omega} K_{\rm S}(\br,\br') \eta(\br') \, \dd s(\br') \; .$$
		Then, for $\br \in \bbR^2$,
		\begin{align}
			\Sthreed \left [  \VS[\mu] \right ](\br) &= \Vphi[\mu](\br) \label{eq:SVS} \; ,\\ 
			\Sthreed \left [ \KS[\eta] \right ](\br) &= \Kphi[\eta](\br) \label{eq:SKS}  \; ,
		\end{align}
		where $\Kphi$ is the integral operator for the kernel 
		$K_{\phi}(\br,\br') = \partial_{\bn'}^{j} \partial_{\btau'}^k \Gphi(\br,\br')$.
	\end{lemma}
	\begin{proof}
		We begin by letting
		\begin{equation*}
			I(\br,\br') = \int_{\bbR^2} \frac{1}{|\br-\tilde\br|}|G_S(\tilde \br, \br')| \, \dd A(\tilde \br)
			\quad \textrm{and} \quad J(\br,\br') = \int_{\bbR^2} \frac{1}{|\br-\tilde\br|}|K_{\rm S}(\tilde \br, \br')| \, 
			\dd A(\tilde \br) \; .
		\end{equation*}
		The value $I(\br,\br')$ exists for all~$\br,\br'$ because $G_S(\tilde \br, \br')=O(|\tilde \br- \br'|^{-3})$ as~$\tilde\br\to\infty$, and is bounded because~$G_S$ has at worst a logarithmic singularity. Similarly, $K_{\rm S}(\tilde \br, \br')$ has similar decay properties and is at worst~$O(|\tilde \br- \br'|^{-1})$ as~$\tilde\br\to \br'$. Then, $J(\br,\br')$ exists
		for $\br\ne \br'$ and standard estimates show that it has at worst a~$\log |\br- \br'|$ singularity as~$|\br- \br'|\to 0$.
		The integrals
		$\int_\Omega I(\br,\br') |\mu(\br')| \, \dd A(\br')$ and $\int_{\partial\Omega} J(\br,\br') |\eta(\br')| \, \dd s(\br')$ are thus finite by the Cauchy-Schwarz inequality. The desired formulas are therefore a consequence of the Fubini-Tonelli theorem and \eqref{eq:Gphiid}.
	\end{proof}

    The following corollaries follow immediately from the proof of the lemma above and the definitions of $\GS$ and $\Gphi$. 
	\begin{corollary}
		In the dissipative regime, for $\mu \in L^2(\Omega),$
		\begin{align}
			\Vphi[\mu](\br) &=  \int_{\mathbb{R}^2} \GS(\br,\br') \int_{\Omega} \frac{1}{4\pi |\brprime-\brprimeprime|} \mu(\brprimeprime) \, {\rm d}A(\brprimeprime) \, {\rm d}A(\brprime) \; ,\\
			\Sthreed [ \Vphi[\mu] ] (\br) &=   p(-\Delta_\br) \Vphi [\mu] (\br) - \int_{\mathbb{R}^2} \frac{1}{4\pi |\br-\brprime |} \mu(\brprime) \, \dd A(\brprime)  \label{eq:svphi} \; .
		\end{align}
		
	\end{corollary}

	\begin{corollary}
    \label{cor:actualsolution}
		Suppose that $d_P \geq 3$, the coefficients are in the dissipative regime, and that $\KS$ is a
		linear combination of boundary integral operators in the same form as \Cref{lem:sgsgphi}. Let
		$\mu \in L^2(\Omega)$, $\eta \in L^2(\partial \Omega)$, and $c\in \bbR$ be given. 
		Then $\sigma(\br) = \mu(\br) + \Vphi[\mu](\br) + c \VS[\mu](\br) + \KS[\eta](\br)$ 
		satisfies $\mathcal{L}[\sigma](\br) = 0$ for $\br \in D \setminus \overline{\Omega}$. 
        Moreover, $\mathcal{L}[\KS[\eta]](\br) = 0$ for $\br \in D \setminus \partial \Omega$.
	\end{corollary}

	\section{The integral representations}\label{sec:reductioninteq}
	Because of the nonlocal nature of \eqref{eq:reducedbvp}, it is convenient to construct a global representation of the solution from a single Green's function. Since we require that the equation be automatically satisfied in the unbounded exterior region $D \, \backslash \, \overline{\Omega}$, we use the exterior Green's function, and its derivatives with respect to the `source' variable. We illustrate this construction with two examples: capillary-gravity waves and flexural-gravity waves. A key feature of our approach is that the surface-boundary portions of the representation can be constructed using ingredients from standard BIE representations for solving PDEs for exterior problems. Roughly speaking, if an integral representation exists for the solution $u$ to the boundary value problem \begin{align*}
		\begin{cases}
			\begin{aligned}
				\Aext[0,u] &=0 \, ,\quad && {\rm in} \,\, D \setminus \overline{\Omega} \, ,\\
				\Bop[0,u] &= g \, ,\quad && {\rm on}\,\,\partial \Omega \, , \\
			\end{aligned}
		\end{cases}
	\end{align*}
	then, by substituting $\GS$ into the integral representation for $u$, we obtain a suitable candidate function $B_j$ in our representation \eqref{eq:ansatz}. The surface-volume portion of the representation $V(\br,\brprime)$ we obtain by taking appropriate linear combinations of $\GS$ and its derivatives.

    The uniqueness of the specific PDE boundary value problems treated below can be established with standard
    arguments in the dissipative regime, assuming certain decay conditions on the 
    potential $\phi$ and its derivatives. It is then possible to establish the uniqueness of solutions
    of the corresponding integro-differential boundary value problem, assuming that $\sigma$ satisfies the
    same radiation condition as $\GS$ (see \Cref{prop:sommerfeld}). Finally, the uniqueness of the 
    solutions of the integro-differential boundary value problem can be used to establish 
    the invertibility of the corresponding integral equation systems, and even their invertibility 
    in the non-dissipative regime by the Gohberg-Sigal theory (with the possible exception of a nowhere
    dense set of parameters). These arguments were taken up in detail previously for
    variable rigidity flexural-gravity wave scattering~\cite{askham2025integral}. For the sake of
    brevity, we simply assume the uniqueness of solutions of the integro-differential boundary 
    value problem when establishing the invertibility of the integral equations below.
    
	\subsection{Exterior capillary-gravity waves}\label{sec:capillarygravity}
	
	The first example we consider is the exterior capillary-gravity wave problem, which arises in the case of porous or partial membranes \cite{yip2001wave,kim1996flexible,manam2012mild,ding2019bragg} and surface films of finite area \cite{dore1982oscillatory, dore1985theory}. In this case, we
	consider a system of the form
	\begin{equation}
		\begin{cases}
			\begin{aligned}
				\Aext [\phi,\partial_z\phi] &:= (-\beta \Delta + \gamma)\partial_z \phi - \phi = 0 \, , && \text{ in } D \setminus \overline{\Omega} \, ,  \\ 
				\Aint [\phi,\partial_z\phi ] &:= \partial_z \phi - \phi = f \, ,  && \text{ in } \Omega \, ,  \\ 
				\Bop_D[\phi,\partial_z\phi] &:= \gamma_0^{+}\partial_z \phi = g \, , \ \text{ or } \ \Bop_N[\phi,\partial_z\phi] := \gamma_1^{+} \partial_z \phi = g \, , && \text{ on } \partial \Omega \, ,
			\end{aligned} 
		\end{cases} \label{eq:extcapbvp}
	\end{equation}
	where $\beta \geq 0$ is related to the surface tension. Two of the common boundary conditions ~\cite{yip2001wave} for the interface 
	$\partial \Omega$ are provided and the appropriate data $g$ depend on the boundary condition
	selected. The Neumann condition $\mathcal{B}_N$ corresponds to a freely floating membrane while the Dirichlet condition $\mathcal{B}_D$ corresponds to a membrane which is fixed using, for instance, a loop or a string.  For ease of exposition, we have set the coefficient in front of the $\partial_z \phi$ term in $\Aint$ to be unity, though the approach in this section applies for any nonzero coefficient after suitable rescaling. 
	
	We now turn to the derivation of the integral equation formulation of this problem. Setting $\phi = \Sthreed[\sigma]$ as before, the system \cref{eq:extcapbvp} becomes
	\begin{equation}
		\begin{cases}
			\begin{aligned}
				\Lext[\sigma] &:= \frac12 (-\beta \Delta + \gamma)\sigma - \Sthreed[\sigma] = 0    \, , && \text{ in } D \setminus \overline{\Omega} \, ,  \\
				\Lint[\sigma] &:= \frac12 \sigma - \Sthreed[\sigma] = f \, , && \text{ in } \Omega \, ,  \\
				\Mop_D[\sigma] &:= \frac{1}{2} \gamma_0^+ \sigma = g \, ,  \ \text{ or } \  \Mop_N[\sigma] := \frac{1}{2} \gamma_1^{+} \sigma = g \, , && \text{ on } \partial \Omega \, .
			\end{aligned}
		\end{cases} \label{eq:capillaryintegrodiff}
	\end{equation}
	The operator $\Lext$ is then in the general form \cref{eq:Lsigma} where $p(z) = \beta z + \gamma$, so that 
	$d_P=3$ in the notation of \Cref{sec:green}. 
	
	The design of our integral representation
	of $\sigma$ is driven by the fact that the Green's function for $\Lext$ is a relatively
	smooth perturbation of the Laplace Green's function, i.e. $\GS + 2 / \beta \GL$ has continuous
	derivatives and its second derivatives are $\log$ singular (see \Cref{lem:smallr}). 
	One immediate consequence is that the jump properties of the boundary layer potentials are analogous 
	to the standard jump properties of Laplace boundary layer potentials.
	\begin{lemma}
		\label{lem:capjumps}
		Let $\GS$ be the Green's function for $\Lext$ and let 
		\begin{align}
			\mathcal{S}_{\rm S}[\eta](\br) &= \int_{\partial \Omega} \GS(\br,\br') \eta(\br')\, \dd s(\br') \; , \label{eq:SSdef} \\
			\mathcal{D}_{\rm S}[\eta](\br) &= \int_{\partial \Omega} \bn(\br')\cdot \nabla_{\br'} \GS(\br,\br') \eta(\br')\, \dd s(\br') \; . \label{eq:DSdef}
		\end{align}
		If $\eta \in L^2(\partial \Omega)$, then for almost every $\br_0\in \partial \Omega$,
		\begin{align*}
			\gamma_0^{\pm} (\mathcal{D}_{\rm S}[\eta])(\br_0) &= \mp \frac{1}{\beta} \eta(\br_0) + 
			\underbrace{\pv  \int_{\partial \Omega}\bn(\br')\cdot \nabla_{\br'} \GS(\br_0,\br') \eta(\br')\, \dd s(\br')}_{=:\mathfrak{D}_{\rm S}[\eta](\br_0)} \; , \\
			\gamma_1^{\pm} (\mathcal{S}_{\rm S}[\eta])(\br_0) &= \pm \frac{1}{\beta} \eta(\br_0) + 
			\underbrace{\pv  \int_{\partial \Omega}\bn(\br_0)\cdot \nabla_\br \GS(\br_0,\br') \eta(\br')\, \dd s(\br')}_{=:\mathfrak{S}'_{\rm S}[\eta](\br_0)} \; .
		\end{align*}
		Moreover, $\SSi[\eta]$ for $\eta \in L^2(\partial \Omega)$ and $\partial_\bn \DS[\eta]$
        for $\eta \in H^{1/2}(\partial \Omega)$ are continuous across the boundary and 
        $\mathfrak{D}_{\rm S},\mathfrak{S}'_{\rm S}:H^{s}(\partial \Omega)\to H^{s+1}(\partial \Omega)$
        for any $s\geq 0$.
	\end{lemma}
	\begin{proof}
		By \Cref{lem:smallr}, $\GS+2/\beta \GL$ has a continuous gradient and its second derivative
        is merely logarithmically singular. The jump results then hold by 
		the standard jump relations for the Laplace Green's function~\cite{kress1999linear,hackbusch2012integral}
        and the regularity results follow from the fact that the normal derivative of the Laplace Green's function is a smooth
        function when restricted to a smooth curve and that operators with kernels containing up to two derivatives of $\GS+2/\beta \GL$ are compact by
        \Cref{lem:gencompact}.
	\end{proof}
	
	Consider the system for the Neumann boundary condition. To obtain a second kind equation on the boundary, a single layer 
	potential is used, by analogy with the Laplace case. For the volumetric term, since convolution with $\GS$ is regularizing, we construct our kernel by taking sufficient derivatives of $\GS$ so that the resulting volumetric portion of the integral equation is second-kind. In particular, in the form of the ansatz \eqref{eq:ansatz}, the kernels $V$ and $B$ we select are:
	\begin{equation}
		V(\br,\brprime) = - \frac{\beta}{2} \Delta_{\brprime} \GS(\br,\brprime)  +  \frac{|\gamma|}{2} \GS(\br,\brprime) \, , \quad 
		B(\br,\brprime) = \beta \GS(\br,\brprime) \, ,
		\label{eq:capillarykernelsN}
	\end{equation}
	where we have added the non-Laplacian term in $V$ for the sake of the invertibility analysis further below.
	It follows immediately from the definition of $\GS$ that
	\begin{equation}\label{eq:capkernV}
		V(\br,\br') = \delta(\br-\br') + \frac{|\gamma|-\gamma}{2} \GS(\br,\brprime) + \Gphi(\br,\brprime) \, .
	\end{equation}
	Note that whenever $\gamma$ is positive this kernel simplifies to $V(\br,\brprime) = \delta(\br-\brprime) + \Gphi(\br,\brprime)$. Applying the interior operator $\Lint$ and boundary operator $\Bop_N$ to this representation results 
	in a system of integral equations of the form \eqref{eq:vie}-\eqref{eq:bie}, where the kernels are given by:
	\begin{align*}
		K_{00}(\br,\brprime) &= \frac{|\gamma|-\gamma}{4}\GS(\br,\brprime) + \frac{1-|\gamma|}{2} \Gphi(\br,\brprime) + \frac{\beta}{2} \Delta_{\brprime} \Gphi(\br,\brprime) \, ,  \\
		K_{10}(\br,\brprime) &=   \frac{|\gamma| - \gamma}{4} \partial_{\bn} \GS(\br,\brprime) + \frac{1}{2} \partial_{\bn} \Gphi(\br,\brprime)  \; , \\
		K_{01} (\br,\brprime) &= \frac{\beta}{2}\GS(\br,\brprime) - \beta \Gphi(\br,\brprime)   \; ,\\
		K_{11} (\br,\brprime) &= \frac{\beta}{2} \partial_{\bn} \GS(\br,\brprime) \; ,
	\end{align*}
	and the identity terms are given by $T_0 = T_1 = \frac{1}{2}.$
	
	For the Dirichlet problem, we may use the same volume potential but the boundary potential 
	is instead a double layer potential, i.e. 
	\begin{equation}
		V(\br,\brprime) = - \frac{\beta}{2} \Delta_{\brprime} \GS(\br,\brprime)  +  \frac{|\gamma|}{2} \GS(\br,\brprime) \, , \quad B(\br,\brprime) = \beta \partial_{\bn'}\GS(\br,\brprime).
		\label{eq:capillarykernelsD}
	\end{equation}
	Then, the kernel $K_{00}$ stays the same, while the rest of the kernels in the integral equation are modified as such:
	\begin{align*}
		K_{10}(\br,\brprime) &=   \frac{|\gamma| - \gamma}{4}  \GS(\br,\brprime) + \frac{1}{2}  \Gphi(\br,\brprime) \, , \\
		K_{01} (\br,\brprime) &= \frac{\beta}{2} \partial_{\bn'}\GS(\br,\brprime) -\beta  \partial_{\bn'} \Gphi(\br,\brprime) \, ,  \\
		K_{11} (\br,\brprime) &= \frac{\beta}{2} \partial_{\bn'} \GS(\br,\brprime) \, ,
	\end{align*}
	and the identity terms are given by $T_0= -T_1 = \frac{1}{2}$. The following pertains to the Fredholm structure of the resulting systems for the Neumann and Dirichlet problems.
	\begin{theorem}\label{thm:capfredholm}
		The system of integral equations given by \cref{eq:vie,eq:bie} is Fredholm second kind on $L^2(\Omega) \times L^2(\partial \Omega)$ for both the Neumann problem with the representation defined by \cref{eq:capillarykernelsN} and the Dirichlet
		problem with the representation defined by \cref{eq:capillarykernelsD}.    
	\end{theorem}
	
	\begin{proof} 
		It suffices to show that the integral operators with kernels $K_{ij},$ $i,j=0,1$ are compact. For the volume-to-volume and
		volume-to-boundary operators, compactness
		follows from \Cref{lem:genregvolume}, the trace theorem, and Rellich's lemma. The boundary-to-volume
		operators are adjoints of volume-to-boundary operators covered by these same results, and thus also
		compact.
		The kernel $K_{11} = \beta \partial_{\bn} G(\br,\br')$ resulting from the Neumann 
		representation is a perturbation of $- \partial_{\bn}\GL(\br,\br')$,
		which is well-known to be smooth on a smooth boundary curve. The difference 
		$K_{11} +  \partial_{\bn}\GL(\br,\br')$ is continuous, and hence the integral operator with kernel
		$K_{11}$ is compact. 
		The reasoning for the Dirichlet case is similar. 
	\end{proof}

	The following result establishes the invertibility of our integral equations.
	
	\begin{theorem} \label{thm:capuniqueness}
	Suppose that the coefficients are in the dissipative regime and that solutions of the integro-differential boundary value problems \cref{eq:capillaryintegrodiff}, supplemented with the decay condition $\sigma(\br) =\mathcal{O}(1/|{\br}|^3)$ as $|\br| \to \infty$, are unique for both the Neumann and Dirichlet cases. Then, the system of integral equations given by \cref{eq:vie,eq:bie} is invertible on $L^2(\Omega) \times L^2(\partial \Omega)$ for both the Neumann problem with the representation defined by \cref{eq:capillarykernelsN} and the Dirichlet problem with the representation defined by \cref{eq:capillarykernelsD}.
	\end{theorem} 
	\begin{proof}
    By the Fredholm alternative and \Cref{thm:capfredholm}, it is sufficient to show the uniqueness of solutions of the 
    integral equation systems.
		Suppose that $\mu \in L^2(\Omega)$ and $\eta \in L^2(\partial\Omega)$ satisfy the homogeneous version of \cref{eq:vie,eq:bie} for either the Dirichlet or Neumann problem. Since the densities $\mu,\eta$ solve the integral equation with zero right-hand side, and the operators in the integral equation are smoothing by \Cref{lem:genregvolume,lem:genregbdry,lem:capjumps}, a standard bootstrapping argument, i.e. iteratively
        applying the regularity results to the integral equations,
        implies that $\mu \in H^1(\Omega)$ and that $\eta \in H^{1/2}(\partial \Omega)$ in the Neumann case and $\eta \in H^{3/2}(\partial \Omega)$ in the Dirichlet case. 
        
        Let $\sigma$ be the corresponding surface density defined by the
        ansatz~\cref{eq:ansatz}, which satisfies the interior equation and boundary condition in \cref{eq:capillaryintegrodiff} by construction and the exterior equation in \cref{eq:capillaryintegrodiff} by \Cref{cor:actualsolution}.
        Uniqueness of the integro-differential equation \cref{eq:capillaryintegrodiff} implies that $\sigma \equiv 0.$
        Recalling the representation of $\sigma$, we have that
        $$ 0 = \sigma(\br) = \mu(\br) + \frac{|\gamma|-\gamma}{2}\VS[\mu](\br) + \Vphi[\mu](\br) + \int_{\partial \Omega} 
        B(\br,\br')\eta(\br') \, \dd s(\br') \; , \quad \br \in \Omega \; .$$
        Re-arranging this expression for $\mu$ and applying the regularity results, we see that in fact $\mu \in H^2(\Omega)$.
        We then apply the exterior operator ($\frac12\lp -\beta \Delta + \gamma\rp - \Sthreed$) to $\sigma$ and evaluate inside $\Omega$ to obtain
		\begin{equation*}
			-\frac\beta2\Delta \mu + \frac{|\gamma|}{2} \mu  = 0 \, , \quad \text{ in } \Omega \, .
		\end{equation*} 
		Therefore, $\mu$ satisfies a homogeneous screened Poisson equation inside $\Omega$.
		
		Now, we examine the Neumann and Dirichlet problems separately. For the Neumann problem, the standard jump relations applied to $\sigma$ and its normal derivative give
		\begin{equation*}
			0 = [[\sigma]] = -\gamma_0^- \mu  \, ,  \qquad
			0 = [[\partial_\bn \sigma]] = -\gamma_1^{-}\mu -  \eta \, .
		\end{equation*}
		The first equation implies that $\mu \equiv 0$ by the standard uniqueness result for the screened Poisson equation with Dirichlet boundary conditions, so that the second equation implies that $\eta \equiv 0$.
		For the Dirichlet problem, we observe that $\eta$ is sufficiently regular that the normal derivative of the double layer $\DScgprime$ is well-defined. By Lemma \ref{lem:capjumps}, $\DScgprime$ is continuous across the boundary and therefore we have the following jumps in $\sigma$:
		\begin{equation*}
			0 = [[\sigma]] = -\gamma_0^- \mu +  \eta \, , \qquad 
			0 = [[\partial_\bn \sigma]] = -\gamma_1^{-}\mu \, .
		\end{equation*}
		The second equation implies that $\mu$ is a solution to the screened Poisson equation with homogeneous Neumann data, thus $\mu \equiv 0$. The first equation then implies that $\eta \equiv 0$.
	\end{proof}
    
	\subsection{Exterior flexural-gravity waves}\label{sec:flexuralgravity}
	
	The second example we consider is the scattering of flexural-gravity waves by an opening in an infinite plate overlying water. In the case when the plate represents an ice sheet, such an opening is known as a polynya \cite{bennetts2010wave}. In this case, the exterior part of the ice sheet supports flexural-gravity waves while its ice-free interior supports ordinary gravity waves:
    \small
	\begin{equation}
		\begin{cases}
			\begin{aligned}
				\Aext [\phi , \partial_z \phi] &:= (\alpha \Delta^2 + \gamma)\partial_z \phi - \phi = 0 \, , && \text{ in } D \setminus \overline{\Omega} \, , \\
				\Aint [\phi , \partial_z \phi] &:= \partial_z \phi - \phi = f \, , && \text{ in } \Omega \, , \\
				\Bop[\phi , \partial_z \phi] &:= \gamma_0^+ \begin{bmatrix}  \displaystyle \nu \Delta  + (1-\nu) {\partial_\bn^2 }  \\ {\partial_\bn^3 } + (2 -\nu) {\partial_\bn \partial_\btau^2 } + (1-\nu) \kappa \left( {\partial_\btau^2 }  - {\partial_\bn^2 } \right) \end{bmatrix} \partial_z \phi = g \, ,  && \text{ on } \partial\Omega \, ,
			\end{aligned}
		\end{cases} \label{eq:extflexbvp}
	\end{equation}
    \normalsize
	where $\alpha > 0$ is related to the flexural rigidity of
        the ice, $\nu$ is Poisson's ratio of ice, $\kappa$ is the signed curvature on the boundary, 
	and $\Bop$ is the standard free plate boundary condition \cite{timoshenko1959theory, landau59}. The first component of $\Bop$ corresponds to a prescribed bending moment at the boundary, while the second component corresponds to a prescribed shear force. For the sake of brevity, we present the integral representation for the case of a simply-connected domain $\Omega$, the multiply-connected case can be treated similarly~\cite{nekrasov2025boundary}.

        Again, setting $\phi = \Sthreed[\sigma]$, this system can be written in terms of $\sigma$ as:
	\begin{equation}
		\begin{cases}
			\begin{aligned}
				\Lext[\sigma] &= \frac{1}{2} (\alpha \Delta^2 +\gamma ) \sigma - \Sthreed[\sigma] = 0   \, , && \text{ in } D \setminus \overline{\Omega} \, ,  \\
				\Lint[\sigma] &= \frac{1}{2}\sigma - \Sthreed[\sigma] = f \, , && \text{ in } \Omega \, , \\
				\mathcal{M}[\sigma] &= \frac12 \gamma_0^+
				\begin{bmatrix}  \displaystyle \nu \Delta  + (1-\nu) {\partial_\bn^2 }  \\ {\partial_\bn^3 } + (2 -\nu) {\partial_\bn \partial_\btau^2 } + (1-\nu) \kappa \left( {\partial_\btau^2 }  - {\partial_\bn^2 } \right) \end{bmatrix} \sigma = g \, , && \text{ on } \partial \Omega \, . \\
			\end{aligned}
		\end{cases} \label{eq:flexuralintegrodiff}
	\end{equation}
	The operator $\Lext$ is then in the general form \cref{eq:Lsigma}, where $p(z) = \alpha z^2 + \gamma$ and $d_P = 5$. The Green's function, $\GS$, corresponding to the exterior operator $\Lext$ is a relatively smooth perturbation of the biharmonic Green's function, i.e. $\GS - 2/\alpha \GB$ has continuous derivatives and its fourth derivatives are $\log$ singular (see \ref{lem:smallr}). The immediate consequence is that the jump properties of these boundary layer potentials for $\GS$ are analogous to the jump properties of the biharmonic Green's function derived in \cite{nekrasov2025boundary}. We briefly review these properties:
	
	\begin{lemma}\label{lem:flexjumps}
		Let $\SSi$ and $\DS$ be defined as in \cref{eq:SSdef,eq:DSdef}, but with $\GS$ corresponding to the exterior flexural gravity wave Green's function. We also define  $\TS$ to be the following boundary operator:
		\begin{equation*}
			\TS[\eta](\br) = \int_{\partial\Omega} \partial_{\btau'} \GS(\br,\brprime) \eta(\brprime) \, \dd s(\brprime) \, .
		\end{equation*}
		For $\eta \in L^2(\partial \Omega)$, these operators and their normal derivatives are continuous across the boundary for almost every $\br_0 \in \partial \Omega$:
		\begin{align*} 
			\gamma_i^\pm ( \SSi[\eta] )(\br_0) &= \int_{\partial \Omega} \partial^i_\bn  \GS (\br,\brprime) \eta(\brprime) \, \dd s(\brprime)    \, , \quad \text{ for } i \in \{ 0,1,2 \} \, , \\
			\gamma_i^\pm ( \DS[\eta] )(\br_0) &=  \int_{\partial \Omega} \partial^i_\bn \partial_{\bn'} \GS(\br,\brprime)  \eta(\brprime) \, \dd s(\brprime)   \, ,  \quad \text{ for } i \in \{ 0,1 \} \, , \\
			\gamma_i^\pm ( \TS[\eta] )(\br_0) &=  \int_{\partial \Omega} \partial_{\bn}^i \partial_{\btau'}  \GS(\br,\brprime)  \eta(\brprime) \, \dd s(\brprime)   \, ,   \quad \text{ for } i \in \{ 0,1,2 \} \, , 
		\end{align*}
		where the integrals above are interpreted in a principal value sense whenever there are three derivatives acting on the Green's function.
		Moreover, we have the following jumps for higher derivatives of these same potentials for almost every $\br_0 \in \partial \Omega$:
		\begin{align*}
			\gamma_3^\pm(\SSi[\eta])(\br_0) &= \pm \frac{1}{\alpha} \eta(\br_0) + \pv \int_{\partial\Omega} \partial^3_\bn  \GS (\br,\brprime) \eta(\brprime) \, \dd s(\brprime)    \, , \\ 
			\gamma_2^\pm ( \DS[\eta] )(\br_0) &= \mp \frac{1}{\alpha} \eta(\br_0) + \pv \int_{\partial\Omega} \partial^2_\bn \partial_{\bn'} \GS(\br,\brprime)  \eta(\brprime) \, \dd s(\brprime) \, .
		\end{align*}
	\end{lemma}
	\begin{proof}
		By \Cref{lem:smallr}, $\GS-2/\alpha \GB$ has three continuous derivatives. The continuity of layer potentials with at most two derivatives follows as a direct consequence. For higher derivatives, the result holds by 
		the jump relations for the biharmonic Green's function derived in Appendix D of~\cite{nekrasov2025boundary}.
	\end{proof}

        As observed at the beginning of this section, the appropriate boundary layer kernel, $B$, can be determined by considering the PDE $\alpha \Delta^2+\gamma=0$ with free plate boundary conditions. Here, our choice of kernel $B$ is taken, by analogy, from~\cite{nekrasov2025boundary}. Meanwhile, the kernel $V(\br,\brprime)$ is constructed in a similar manner to \Cref{sec:capillarygravity}. In particular, we set
	\begin{align}
		V(\br,\brprime) &=  \frac{\alpha}{2} \Delta^2_{\brprime} \GSfg(\br,\brprime)  \, , \label{eq:flexuralkernel1} \\
		B(\br,\brprime) &= \begin{bmatrix}
			\partial_{\bn'} \GS(\br,\brprime) + \lambda \, (\partial_{\btau'} \GS (\br,\brprime) \star K_{\mathcal{H}} (\br,\brprime) )  (\br,\brprime) \\ 
			\GS(\br,\brprime)
		\end{bmatrix} \, , \label{eq:flexuralkernel2}
	\end{align}
	where $\lambda = (1+\nu)/2$, $\GS$ is the Green's function corresponding to the exterior operator $\Lext$, $K_{\mathcal{H}}$ is the kernel of the Hilbert transform, which is given by
	\begin{equation*}
		K_{\mathcal{H}} := \frac{1}{\pi} \frac{(\br - \brprime) \cdot \btau(\brprime)}{|\br-\brprime|^2} \, ,
	\end{equation*}
	and the $(\star)$ symbol denotes the convolution operation on the boundary, defined by
	\begin{equation*}
		(L \star K)(\br,\brprime) = \int_{\partial\Omega} L(\br,\brprimeprime) K(\brprimeprime,\brprime) \, \dd s(\brprimeprime) \, .
	\end{equation*}
	It follows immediately from the definition of $\GS$ that
	\begin{equation*}
		V(\br,\brprime) = \delta(\br,\brprime) - \frac{\gamma}{2} \GS(\br,\brprime) + \Gphi(\br,\brprime)
	\end{equation*}
        
	As before, the jump properties of the boundary operators can be
        determined from the corresponding properties for the analogous
        PDE kernels:
	\begin{lemma}
		\label{lem:freejumps}
		Let $\Omega$ be a simply-connected domain, let $\eta_1,\eta_2 \in L^2(\partial \Omega)$, and let $\Kop_1[\eta_1] = \DS [\eta_1] + \lambda \TS \mathcal{H}[\eta_1]$, where $\lambda = (1+\nu)/2$, and $\Kop_2[\eta_2] = \SSi[\eta_2]$ be defined as the boundary integral operators corresponding to $B_1(\br,\brprime)$ and $B_2(\br,\brprime)$, respectively. The exterior limits of the boundary conditions applied to these layer potentials can be written, for $\br_0\in\partial \Omega$, as
        \small
		\begin{equation*}
			\gamma_0^+(\mathcal{M}[\Kop_1[\eta_1] \ \Kop_2[\eta_2]] )(\br_0) = \begin{bmatrix}
				\displaystyle\frac{-1+\lambda^2}{\alpha} \eta_1(\br_0) + \mathcal{K}_{11}[\eta_1](\br_0) & \Kop_{12}[\eta_2](\br_0) \\
				\mathcal{K}_{21}[\eta_1](\br_0) & \displaystyle\frac{1}{\alpha}\eta_2(\br_0) + \Kop_{22}[\eta_2](\br_0)
			\end{bmatrix} \; ,
		\end{equation*}
        \normalsize
                where the kernels of the integral operators $\Kop_{11}$, $\Kop_{12}$, $\Kop_{21}$, and $\Kop_{22}$ are given by
                \small
		\begin{align*}
			K_{11}(\br,\brprime) &= \frac12 \partial_\bn^2 \partial_{\bn'} \GS(\br,\brprime) + \frac{\lambda}{2} \, ((\partial_\bn^2  \partial_{\btau'} \GS (\br,\brprime) - \frac{2}{\alpha} K_{\mathcal{H}}(\br,\brprime)) \star K_{\mathcal{H}} (\br,\brprime) )  (\br,\brprime) \\ &\quad +\frac{\nu}{2} \partial_\btau^2 \partial_{\bn'} \GS(\br,\brprime) + \frac{\lambda \nu}{2} ( ( \partial_\btau^2 \partial_{\btau'} \GS (\br,\brprime)  - \frac{2}{\alpha} K_{\mathcal{H}}(\br,\brprime) ) \star K_{\mathcal{H}} (\br,\brprime) )  (\br,\brprime)
			\\
			&\quad - \frac{4\lambda^2}{\alpha} (K_\Dop(\br,\brprime) \star K_\Dop(\br,\brprime))(\br,\brprime) \; , \\
			K_{12}(\br,\brprime) &= \frac{1}{2} \partial^2_\bn \GS(\br,\brprime) + \frac{\nu}{2} \partial^2_\btau \GS(\br,\brprime) \; ,\\
			K_{21}(\br,\brprime) &= \frac12 {\partial_\bn^3 } \partial_{\bn'} \GS(\br,\brprime) + \frac{2 -\nu}{2} {\partial_\bn \partial_\btau^2 } \partial_{\bn'} \GS(\br,\brprime) + \frac{1-\nu}{2} \kappa \,  {\partial_\btau^2 } \partial_{\bn'} \GS(\br,\brprime)  \\
			&\quad - \frac{1-\nu}{2} \kappa \, {\partial_\bn^2 } \partial_{\bn'} \GS(\br,\brprime)  + \frac{\lambda}{2}  ({\partial_\bn^3 } \partial_{\btau'} \GS (\br,\brprime) \star K_{\mathcal{H}} (\br,\brprime) )  (\br,\brprime) \\
			&\quad + \frac{2 -\nu}{2}  \lambda \, ({\partial_\bn \partial_\btau^2 } \partial_{\btau'} \GS (\br,\brprime) \star K_{\mathcal{H}} (\br,\brprime) )  (\br,\brprime) \\
			&\quad + \frac{1-\nu}{2} \lambda \kappa \left( {\partial_\btau^2 }  - {\partial_\bn^2 } \right)  \, (\partial_{\btau'} \GS (\br,\brprime) \star K_{\mathcal{H}} (\br,\brprime) )  (\br,\brprime) - \frac{\lambda}{\alpha} K_{\mathcal{H}'}(\br,\brprime) \; , \\
			K_{22}(\br,\brprime) &= \frac12 {\partial_\bn^3 } \GS(\br,\brprime) + \frac{2 -\nu}{2} {\partial_\bn \partial_\btau^2 } \GS(\br,\brprime) + \frac{1-\nu}{2} \kappa \left( {\partial_\btau^2 }  - {\partial_\bn^2 } \right) \GS(\br,\brprime) \; ,
		\end{align*}
        \normalsize
		and $K_\Dop$ is the kernel of the Laplace double layer potential in two dimensions, i.e.
		\begin{equation*}
			K_\Dop(\br,\brprime) = \frac{(\br -\brprime ) \cdot  \bn (\brprime) }{{2\pi} |\br -\brprime |^2} \, .
		\end{equation*}
	\end{lemma}
	\begin{proof}
		By \Cref{lem:smallr}, $\GS-2/\alpha \GB$ has three continuous derivatives. The result then holds by 
		the jump relations and limits for the biharmonic Green's function derived in~\cite{nekrasov2025boundary}.
	\end{proof}

        \begin{remark}
	  The kernel of the Laplace double layer $K_\Dop$ appears by subtracting a multiple of the Hilbert transform kernel from $\partial_\bn^2 \partial_{\btau'}\GS$ and $\partial_\btau^2 \partial_{\btau'}\GS$ and applying the generalized Poincar\'e-Bertrand formula \cite{muskhelishvilisingularbook,shidong}.
        \end{remark}

        \Cref{lem:freejumps} provides the lower right block of the system of integral equations \eqref{eq:vie}-\eqref{eq:bie}. The rest of the blocks can be obtained by applying the interior operator $\Lint$ to $V$ and $B$ and boundary operator $\Mop$ to $V$:
	\begin{align*}
		K_{00}(\br,\brprime) &= - \frac{\gamma}{4} \GS(\br,\brprime) + \frac{1}{2} \Gphi(\br,\brprime) - \frac{\alpha}{2} \Delta^2_{\brprime} \Gphi(\br,\brprime) \\ 
		K_{01}(\br,\brprime) &=  \frac12 \partial_{\bn'} \GS(\br,\brprime) - \partial_{\bn'} \Gphi(\br,\brprime) + \frac{\lambda}{2} \, (\partial_{\btau'} \GS (\br,\brprime) \star K_{\mathcal{H}} (\br,\brprime) )  (\br,\brprime) \\
		&\quad - \lambda \, (\partial_{\btau'} \Gphi (\br,\brprime) \star K_{\mathcal{H}} (\br,\brprime) )  (\br,\brprime) \\
		K_{02}(\br,\brprime) &= \frac{1}{2} \GS(\br,\brprime) - \Gphi(\br,\brprime) \\
		K_{10}(\br,\brprime) &= - \frac{\gamma}{4}  \partial_\nr^2 \GS(\br,\brprime) + \frac12 \partial_\nr^2 \Gphi(\br,\brprime)  - \frac{\gamma\nu}{4}  \partial_\taur^2 \GS(\br,\brprime) + \frac{\nu}{2} \partial_\taur^2 \Gphi(\br,\brprime) \\
		K_{20}(\br,\brprime) &=  -\frac{\gamma}{4} {\partial_\bn^3 } \GS(\br,\brprime) - \frac{2 -\nu}{4} \gamma {\partial_\bn \partial_\btau^2 }  \GS(\br,\brprime) -  \frac{1-\nu}{4} \gamma \kappa \left( {\partial_\btau^2 }  - {\partial_\bn^2 } \right)  \GS(\br,\brprime)    \\
		& \quad  + \frac12 {\partial_\bn^3 }   \Gphi(\br,\brprime) +\frac{2 -\nu}{2} {\partial_\bn \partial_\btau^2 }   \Gphi(\br,\brprime) + \frac{1-\nu}{2} \kappa \left( {\partial_\btau^2 }  - {\partial_\bn^2 } \right)   \Gphi(\br,\brprime)
	\end{align*}
	where the identity term $T_{0} = \frac{1}{2}$. Using \Cref{lem:freejumps}, the system of integral equations can be summarized as follows:
	\begin{align}
		\begin{pmatrix}
			\frac12 I & 0 & 0 \\
			0 & \frac{-1+\lambda^2}{\alpha} I & 0 \\
			0 & 0 & \frac{1}{\alpha}I
		\end{pmatrix}  \begin{pmatrix}
			\mu \\ \eta_1 \\ \eta_2
		\end{pmatrix} +  \begin{pmatrix}
			\mathcal{K}_{00} & \mathcal{K}_{01} & \mathcal{K}_{02}  \\
			\mathcal{K}_{10} & \mathcal{K}_{11} & \mathcal{K}_{12} \\
			\mathcal{K}_{20} & \mathcal{K}_{21} & \mathcal{K}_{22} \\
		\end{pmatrix} \begin{pmatrix}
			\mu \\ \eta_1 \\ \eta_2
		\end{pmatrix} = \begin{pmatrix}
			f \\ g_1 \\ g_2
		\end{pmatrix} \, , \label{eq:flexIEs}
	\end{align}

	\begin{theorem} \label{thm:flexfredholm}
		The system of integral equations given by \eqref{eq:flexIEs} is Fredholm second kind on $L^2(\Omega) \times (L^2(\partial \Omega))^2$. 
	\end{theorem}
	The proof follows from relatively standard arguments. See \Cref{thm:capfredholm} for more details.
	The following result establishes the invertibility of our integral equation:
	\begin{theorem} \label{thm:flexuniqueness}
	  Suppose that the coefficients are in the dissipative regime and
          that solutions of the integro-differential boundary value
          problem \cref{eq:flexuralintegrodiff}, supplemented with the
          decay condition $\sigma(\br) = \mathcal{O}(1 / | \br |^3 )$ as
          $|\br| \to \infty$, are unique. Then, the system of integral
          equations given by \cref{eq:flexIEs} is invertible on $L^2(\Omega) \times L^2(\partial \Omega)^2$.
	\end{theorem}
	\begin{proof}
		
	  The proof is similar to that of \Cref{thm:capuniqueness}.
          By the Fredholm alternative and \Cref{thm:flexfredholm}, it is sufficient to show the uniqueness of solutions of the 
    integral equation system.
	  Suppose that $\mu \in L^2(\Omega)$ and $\eta_1,\eta_2 \in L^2(\partial\Omega)$ solve the homogeneous version of \cref{eq:flexIEs}. Bootstrapping as in the proof
          of \cref{thm:capuniqueness}, we obtain that $\mu \in H^1(\Omega) $, $\eta_1 \in H^{3/2}(\partial \Omega)$, and
          $\eta_2 \in H^{1/2}(\partial \Omega)$.
          Let $\sigma$ be the corresponding surface density defined by
          the ansatz~\cref{eq:ansatz}, which satisfies the interior
          equation and boundary condition in \cref{eq:flexuralintegrodiff} by construction and the exterior equation in \cref{eq:flexuralintegrodiff} by \cref{cor:actualsolution}.
          Uniqueness of the integro-differential equation implies that $\sigma \equiv  0$.  

            Recalling the representation of $\sigma$, we have that 
            $$ 0 = \sigma(\br) = \mu(\br) - \frac{\gamma}{2}\VS[\mu](\br) + \Vphi[\mu](\br) + \sum_{j=1}^2 \int_{\partial \Omega} B_j(\br,\br') \eta_j(\br') \, \dd s(\br') \; , \quad \br \in \Omega \; .$$
            Re-arranging this equation for $\mu$ and applying the regularity results, we obtain $\mu\in H^4(\Omega)$.
          We then apply the exterior operator ($\frac12\lp \alpha \Delta^2 + \gamma\rp - \Sthreed$) to $\sigma$ and evaluate inside
          $\Omega$ to obtain
		\begin{equation*}
			\frac{\alpha}{2}\Delta^2 \mu = 0 \, , \quad  \text{ in } \Omega \, .
		\end{equation*} Lemma \ref{lem:flexjumps} implies the following jumps in $\sigma$ and its normal derivative across the boundary:
		\begin{equation*}
			0 = [[\sigma]] = - \gamma_0^-\mu  \, , \qquad
			0 = [[\partial_\bn \sigma]] = - \gamma_1^-\mu \, . 
		\end{equation*}
		Therefore, $\mu \equiv 0$ by the standard uniqueness theorem for the interior biharmonic equation with clamped boundary conditions. To show that the boundary densities $\eta_1$ and $\eta_2$ are zero, we must also look at higher order jumps in the solution. By \Cref{lem:flexjumps}, the second normal derivative has the following jump across the boundary:
		\begin{align*}
			0 = [[\partial_\bn^2 \sigma]] &= - \frac{2}{\alpha} \eta_1 - \gamma_2^-\mu  \, ,
		\end{align*}
		which implies that $\eta_1 \equiv 0$. Similarly, the jump $[[\partial_\bn^3 \sigma]]=0$ implies that $\eta_2 \equiv 0$.
	\end{proof}

        \begin{remark}
          The bootstrapping argument mentioned in the proof above
          requires certain regularizing properties of the integral
          operators $\mathcal{K}_{ij}$ in \cref{eq:flexIEs}. These
          can be established using results for more standard integral
          kernels and the explicit cancellations between terms that
          are detailed in \cite{nekrasov2025boundary}.
        \end{remark}
    
	\section{Numerical implementation and scalability}

    While the kernels that appear in the integral equations above 
    are non-standard, there has been significant progress in recent decades in
    {\em kernel independent} methods for the high-order accurate 
    discretization of singular integrals and the fast solution of linear systems
    with certain rank structures. We briefly describe a scalable numerical scheme for solving these integral equations using such methods in the 
    sections below.

    \begin{remark}
        Both for the flexural and capillary problems, the interior normal derivative of the surface-volume density $\mu$ (and of the solution $\partial_z\phi$ itself) has a logarithmic singularity at the boundary. Such behavior is expected due to the presence of $1/r$-type surface-volume terms in our integral equations. In our numerical methods, we partially alleviate this issue by simply refining adaptively near the boundary. In principle, special discretization methods could be developed which capture these singularities with many fewer degrees of freedom. We note that a straightforward consequence of Lemma \ref{lem:genregvolume} is that for smooth data the surface-volume densities are smooth away from the boundary.
    \end{remark}
	\subsection{Discretization of integral operators}
	\label{sec:discretization}
    We numerically solve the integral equations of \Cref{sec:capillarygravity,sec:flexuralgravity} using
    a collocation scheme~\cite{Greengard2021FMM}. The boundary curve $\partial \Omega$ is initially discretized by a 
    16th-order panelization and the interior of $\Omega$ is represented as an 8th-order curved triangular
    mesh, with each panel of the boundary corresponding to an edge of a triangular element. The discretization
    nodes on each triangular element are scaled ``Vioreanu-Rokhlin'' nodes~\cite{vioreanu2014spectra}, which 
    are stable for high-order interpolation in the total degree polynomial basis, and the nodes on the boundary
    are scaled Legendre nodes; see \Cref{fig:cactus} for an illustration. 

            \begin{figure}[h]
    \centering
    \includegraphics[width=\linewidth]{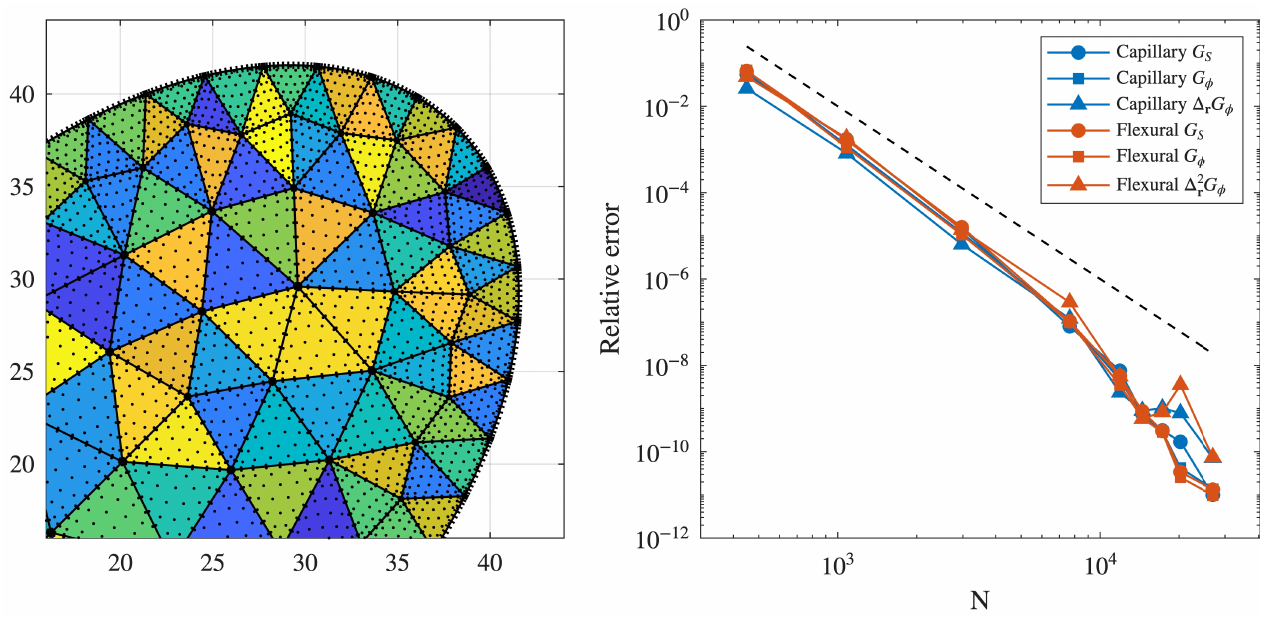}
    \vspace{-0.7cm}
    \caption{Discretization of a geometry (left) and convergence of surface-volume operators (right). The geometry uses 8th-order Vioreanu-Rokhlin nodes inside and 16th-order Gauss-Legendre panels on the boundary. The integral operators were applied to a smooth density on a disk and compared to a reference value for a point in the domain. The dashed line represents 8th-order convergence.}
    \label{fig:cactus}
\end{figure}
    
    To discretize the integral
    operators, which are generally weakly singular, we apply special tables of high-order accurate quadrature 
    rules designed for general kernels with known singularity type for targets interior to a patch, we 
    apply adaptive quadrature for targets on the boundary of a patch or near a patch, and we apply a fixed,
    high-order quadrature rule for well-separated targets.
    Letting $\{\br_i\}_1^N$ be the nodes of the triangular mesh, this quadrature method produces 
    a discretization of the {\em surface-volume to surface-volume} operators of the form 
\begin{equation}
\label{eq:quadsetup}
\int_{\Omega} K(\br_i,\br') \mu(\br') \, \dd A(\br') \approx \sum_{j\in J_i} w_{ij} \mu_j + 
\sum_{j \not\in J_i} w_j K(\br_i,\br_j) \mu_j \; , \quad i=1,\ldots,N \, ,
\end{equation}
where $K$ is some integral kernel and $J_i$ is a set of indices of volume nodes that are sufficiently 
close to $\br_i$ to require special quadrature. The other operators have similar discrete forms.

\begin{remark}
In practice, in packages such as \texttt{fmm3dbie}, the interactions between far away points, i.e. the second sum in \cref{eq:quadsetup}, is done using an upsampled quadrature rule to control the size of $J_i$; see~\cite{Greengard2021FMM}, for a detailed discussion.
\end{remark}

To generate the quadratures for the weakly singular integrals over the triangular mesh, we use
the \texttt{fmm3dbie} package \cite{fmm3dbie,Greengard2021FMM}, which applies the special 
quadrature rules in~\cite{bremer2012nystrom,Bremer2013SingularIntegrals,Xiao2010QuadratureRules} for $1/|\br|$ type singularities,
and a similar procedure for $\log|\br|$ type singularities.
We use the package \texttt{chunkIE} \cite{chunkIE} to generate the quadratures 
for integrals over the panelization, which applies special quadrature rules generated
using the method of~\cite{bremer2010nonlinear}. Both packages use orthogonal polynomial bases
for interpolation and adaptive quadrature for nearly singular integrals.

For a smooth density, the resulting quadrature rules have an order of accuracy determined by the interpolation
order of accuracy, i.e. 8th-order for the triangular mesh and 16th-order for the panelization.
We provide a convergence plot for this scheme in \Cref{fig:cactus}, where the
discretized operators were applied to a smooth test density (a sum of Gaussians) on a disk 
and interpolated from the mesh nodes to a fixed point. The mesh was refined approximately uniformly and
the values from the discretized operators were compared to 
a high precision value obtained from Matlab's adaptive integration routine \texttt{integral2}. 
As noted above, it is not expected that the density $\mu$ will be smooth near the boundary.
To handle this, we discretize the geometry with extra levels of refinement near the boundary,
and the observed order of convergence of the overall scheme is lower than that of the quadrature rule.
    
	\subsection{Scaling to large examples}
	\label{sec:algorithm}
	If the maximum triangle diameter and panel length is bounded by~$h$, then the naive dense application of the system matrix will take~$O(h^{-4})$ time, and a dense solve will take $O(h^{-6})$ time. In order to avoid this, we use the precorrected-FFT method~\cite{phillips2002precorrected,nie2002fast,bruno2001fast,yan2011efficient,li2017precorrected} to apply our system matrix in~$O(h^{-2}\log(h^{-1}))$ time and solve the system iteratively using GMRES~\cite{saad1986gmres}. In short, the precorrected-FFT method in two dimensions is a method for computing $N$-body calculations:
	\begin{equation}\label{eq:Nbody}
		u_i = \sum_{j\neq i} K(\bx_i-\bx_j)\zeta_j,\quad i = 1,\ldots
	\end{equation}
	where $K$ is a translationally invariant kernel that is smooth away from $0$, the $\zeta_i$ are arbitrary complex numbers, and the $\bx_j$ are a collection of points in $\bbR^2$. In this work, $K$ will be~$\GS$, $G_\phi$, or some combination of their derivatives. The method is based on the observation that if the points~$\bx_i$ happen to be located on an equispaced grid, then~\eqref{eq:Nbody} is a convolution and thus can be computed efficiently using the FFT. In order to take advantage of this observation, let~$\bz_i$ be an equispaced grid covering $\Omega$. For any~$\by,$ let $X_{\by,n}$ be the $n\times n$ subset of the grid centered near $\by.$ For a given 
    $r>0$, let $\chi_i(\by)$ be equivalent charges such that
	\begin{equation}\label{eq:pre_cor_approx}
		K(\bx-\by) \approx \sum_{\bz_i \in X_{\by,n}} K(\bx-\bz_i) \chi_i(\by),
	\end{equation}
	for all~$\bx\in \bbR^2\setminus B_{r}(\by)$ with $X_{\by,n} \subset \overline{B_{r}(\by)}$. Similarly, let $\psi_i(\bx)$ denote effective weights so that
	\begin{equation}\label{eq:pre_cor_approx2}
		K(\bx-\by) \approx \sum_{\bz_j \in X_{\bx,n}}\sum_{\bz_i \in X_{\by,n}} K(\bz_j-\bz_i) \psi_j(\bx)\chi_i(\by),
	\end{equation}
    holds for $\bx$ and $\by$ satisfying $B_{r}(\bx) \cap B_{r}(\by) = \emptyset$. Returning to the $N$-body calculation, let $$Q_i = \{ \bx_{j} | B_{r}(\bx_{i}) \cap B_{r}(\bx_{j}) \neq \emptyset\}.$$ 
    We can then write
	\begin{align}
		u_i \approx &\sum_{\bz_{\ell} \in X_{\bx_i,n}} \psi_\ell(\bx_i) \left(\sum_{\bz_k \neq \bz_\ell} K(\bz_\ell - \bz_k) \left(\sum_{j \in \{j\,|\, \bz_k \in X_{\bx_j,n}\}} \zeta_j\chi_k(\bx_j) \right)\right) \label{eqn:precor}\\
		\nonumber &\quad + \sum_{j \in Q_i} \zeta_j \left[ K(\bx_i - \bx_j) - \sum_{\bz_k \in X_{j,n}}\sum_{\bz_\ell \in X_{i,n}} K(\bz_k-\bz_\ell)\psi_\ell(\bx_i)\chi_k(\bx_j) \right]. 
	\end{align}
	The first term in \eqref{eqn:precor} computes the sum as if the expansion \eqref{eq:pre_cor_approx2} were valid between every pair of source and target points $(\bx_i,\bx_j).$ The second term in \eqref{eqn:precor} corrects the error in the first sum for sources and targets which are too close by subtracting the contribution from the first sum for these points, and adding the correct contribution. The advantage of this decomposition is that the first term in \eqref{eqn:precor} can be computed by applying a sparse matrix with entries determined by the $\chi_k$ to the vector of charges $\vec{\zeta} =\{\zeta_j\}_1^N,$ applying a 2D-Toeplitz matrix (the sum over the equispaced grid points), and applying a second sparse matrix with entries given by the $\psi_\ell.$ The second term in \eqref{eqn:precor} corresponds to the application of a sparse matrix to $\vec{\zeta}.$ The sparse matrices have $O(N)$ non-zero entries, assuming the $|Q_i|$ are bounded independent of $N$, and the 2D-Toeplitz matrix may be applied in $\mathcal{O}(N \log{N})$ work, assuming the equispaced grid has $O(N)$ points.

    \begin{figure}
		\centering
		\includegraphics[width=0.3\linewidth]{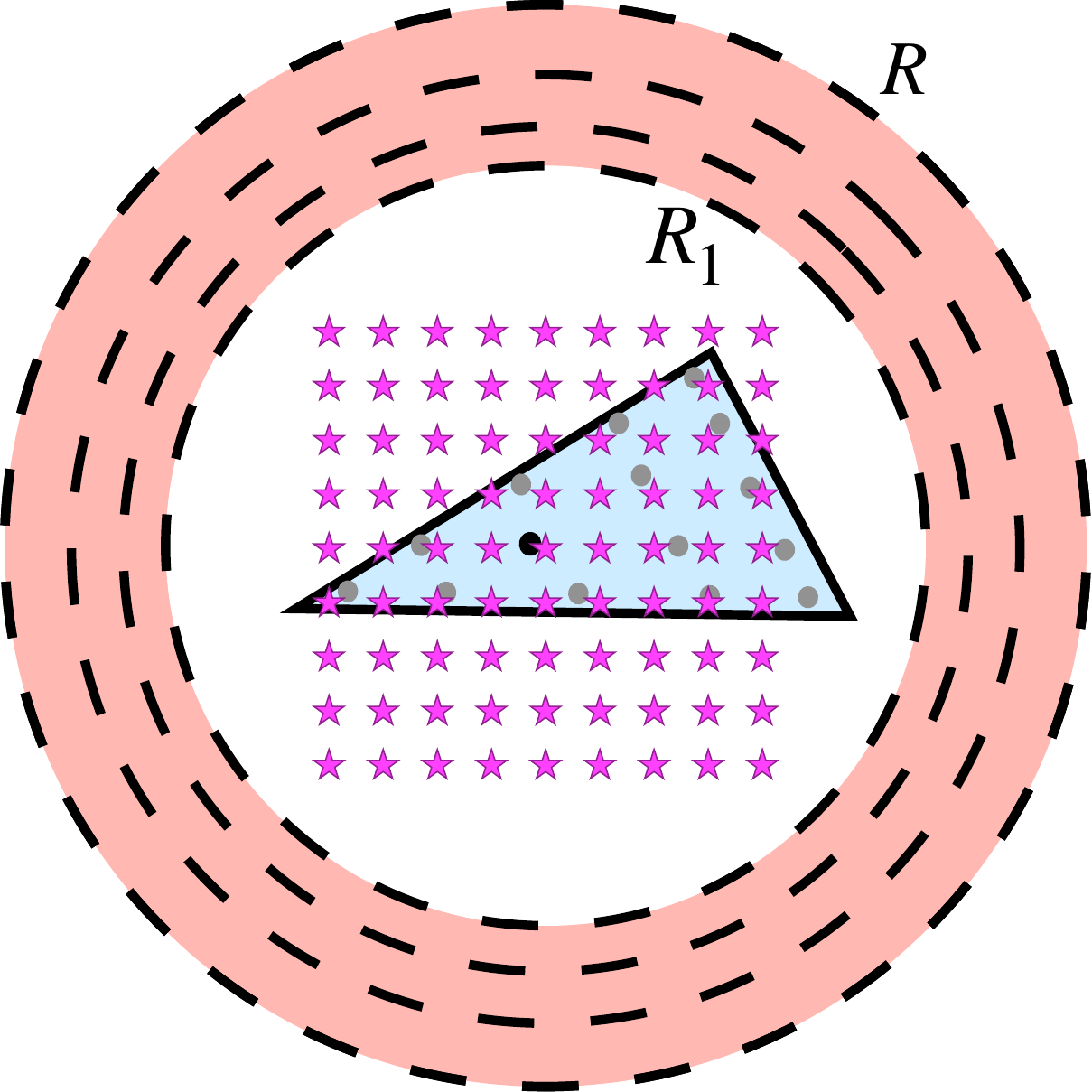}
		\caption{The geometry for the spreading algorithm. A given source (black dot) is spread to the
        nearest $n\times n$ subset of the equispaced grid (pink stars). The equivalent charges on the $n\times n$ grid are chosen 
        so that the fields agree on a series of proxy rings (salmon rings).}
		\label{fig:precorrected}
	\end{figure}

	To compute the coefficients $\chi_k,\psi_k$ we use a proxy point method \`{a} la \cite{xing2020interpolative,ye2020analytical,minden2017fast}. Here we focus on the computation of $\chi_k;$ $\psi_k$ can be calculated using an almost identical procedure. For a given source, $\by,$ fix an annulus with center at the center of $X_{\by,n}$, inner radius $R_1$, and outer radius $R,$ taking $R_1$ sufficiently large that $X_{\by,n}$ is contained within the inner circle;
    see \Cref{fig:precorrected} for an illustration. Consider $C$ circles in the annulus, with radii $r_1=R_1,\ldots,r_C=R$, chosen so that~$1/r_m$ are equally spaced, and suppose that each circle is discretized with $M$ equispaced points. Denote these $MC$ points, called \emph{proxy points,} by ${\bf w}_\ell.$ Suppose ${\bf c}$ is the least squares solution to $A {\bf c}= {\bf b}$, where the entries of the matrix $A \in \mathbb{C}^{MC \times n^2}$, and $b \in \mathbb{C}^{MC}$ are given by
    $$
        A_{\ell,k} = K({\bf w}_\ell - \bz_{j_k}) \,,\quad b_{\ell} = K({\bf w}_{\ell} - \by) \,.
    $$
    Here $\bz_{j_k}, k=1,\cdots,n^2$ are the gridpoints in $X_{\by,n}.$ The coefficients $\chi_j(\by)$ are given by $\chi_{j_k}(\by) = c_k,$ $k=1,\cdots, n^2$ with all other coefficients set to zero.  The coefficients $\chi_k,\psi_k$ can be precomputed and reused across GMRES iterations, and the precomputation step requires $\mathcal{O}(N)$ work. Observe that the matrix $A$ is the same for each source and its pseudo-inverse can be computed once and re-used to compute the coefficients across sources. 
    
     We select the parameters $R_1,R,M,C,n$ and the grid spacing, $\delta z$, heuristically (see, e.g., \cite{xing2020interpolative}).
    In general, if $R,M,C$ and $n$ are sufficiently large, then the error in \eqref{eq:pre_cor_approx} can be made arbitrarily small for any $r\geq R_1 + \sqrt{2} \delta z/2$. In this work, we set~$M=91$; $R_1$ to be $n-1$ times the grid spacing; $R$ to be the larger of $1.5 R$ and $\frac{2\pi}{\rho_1}$, where~$\rho_1$ is the largest real root of~$P$; and~$n=11$. The grid spacing is chosen in proportion to the diameter of the smallest element, and to balance the cost of the two terms in~\eqref{eqn:precor}. We note that for kernels arising from second-order elliptic PDEs, only one circle is required ($C=1$)~\cite{cheng2005compression}. For integro-differential equations more circles are typically required, though empirically in our numerical experiments $C=5$ is sufficient. 
    Further details on the effects of the different parameters will be discussed in a subsequent paper. We summarize the breakdown of the total cost of applying the slowest operator, the flexural $\Vphi$, into the time for the quadrature-correction generation (QG), the precomputation time for the pre-corrected FFT (PC), and the time for a single application of the operator using the precomputed quantities (A) in \Cref{table:timings}. We emphasize that the total computational cost for the full solve will be the sum of QG and PC, together with A multiplied by the total number of GMRES iterations.

\begin{table}[h!]
\centering
\begin{minipage}{0.25\linewidth}
    \centering
    \includegraphics[width=\linewidth]{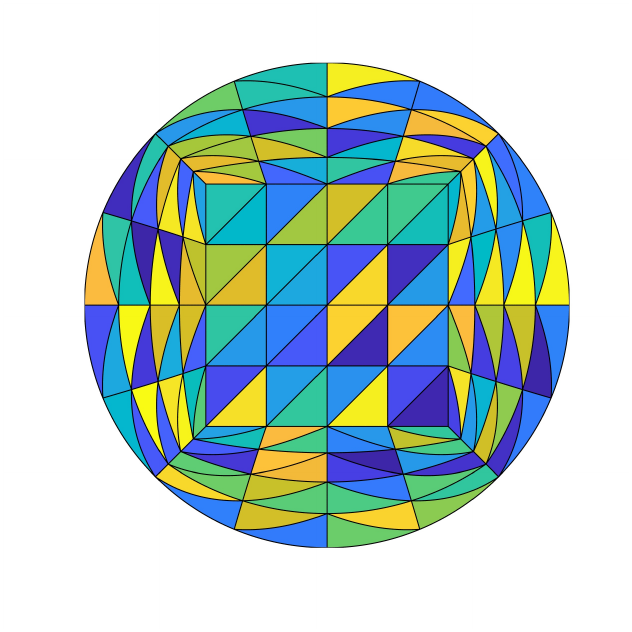}
\end{minipage}%
\hspace{0.01\linewidth}
\begin{minipage}{0.7\linewidth}
    \centering
    \begin{tabular}{|| c | c c c c||} 
     \hline
     N & 450 & $4\cdot 450$ & $4^2\cdot 450$ & $4^3\cdot 450$ \\ [0.1ex] 
     \hline 
     QG & 3.28 s & 12.5 s & 56.3 s & 231 s \\ 
     PC & 3.95 s & 8.97 s & 33.6 s & 128 s \\
     A  & 0.00103 s & 0.00199 s & 0.00542 s & 0.0186 s \\ [0.1ex] 
     \hline
    \end{tabular}
\end{minipage}
\caption{Time for quadrature generation (QG), precomputation (PC), and application (A) on a circular domain as the total number of points (N) increases. The mesh for $N = 4^2 \cdot 450$ is displayed on the left.}
\label{table:timings}
\end{table}

	\section{Examples and applications}
	\label{sec:examples}
	For scattering problems, we consider a total velocity potential given as the sum
	of an incident and a scattered velocity potential, $\phitot = \phiinc + \phi$. For the examples considered here, the incident field is given as a plane wave $\phiinc = \exp(i \mathbf{k}\cdot \br) \exp( - k z )$ where $k = |\mathbf{k}|$ is the real root $\rho_1$ of the dispersion relation in the exterior region. Due to the nonlocal nature of the boundary value problems, simple analytic solutions are difficult to construct in generic domains, therefore the error in the solution was checked by `self-convergence', i.e. solving the integral equation on both $6^{\rm th}$ and $8^{\rm th}$ order discretizations and comparing the results. The $8^{\rm th}$ order solution was computed on the $6^{\rm th}$ order mesh by evaluating the integral representation for $\partial_z \phi$ and by evaluating the formula (\ref{eq:vie}) for $\mu$.  Solutions and errors are shown for capillary-gravity waves with $\beta = 0.5, \gamma = 1$ in Figure \ref{fig:caperrorstar} and for flexural-gravity waves with $\alpha = 1.5, \gamma = -0.1, \nu = 0.3$ in Figure \ref{fig:flexerrorstar}. Inside the domain, the pointwise error in the density $\mu$ is plotted, while in the exterior the pointwise error in $\partial_z \phi$ is plotted. Both errors were normalized by the maximum value of $|\partial_z \phi|$.

    \begin{figure}[h]
		\centering
		\includegraphics[width=0.8\linewidth]{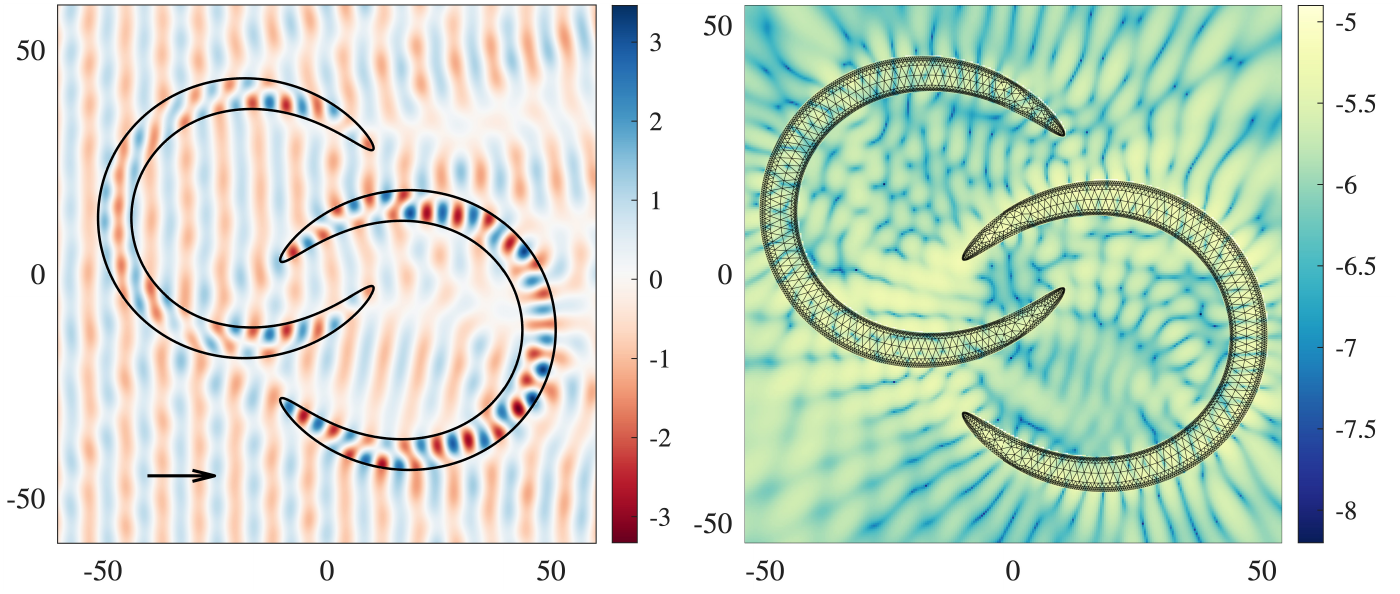}
		\caption{Exterior capillary-gravity waves with Neumann boundary conditions. The real part of $\partial_z \phi$ is plotted on the left, while the $\log_{10}$ relative self-convergence error is plotted on the right.  }
		\label{fig:caperrorstar}
	\end{figure} 
    
	\begin{figure}
		\centering
		\includegraphics[width=0.8\linewidth]{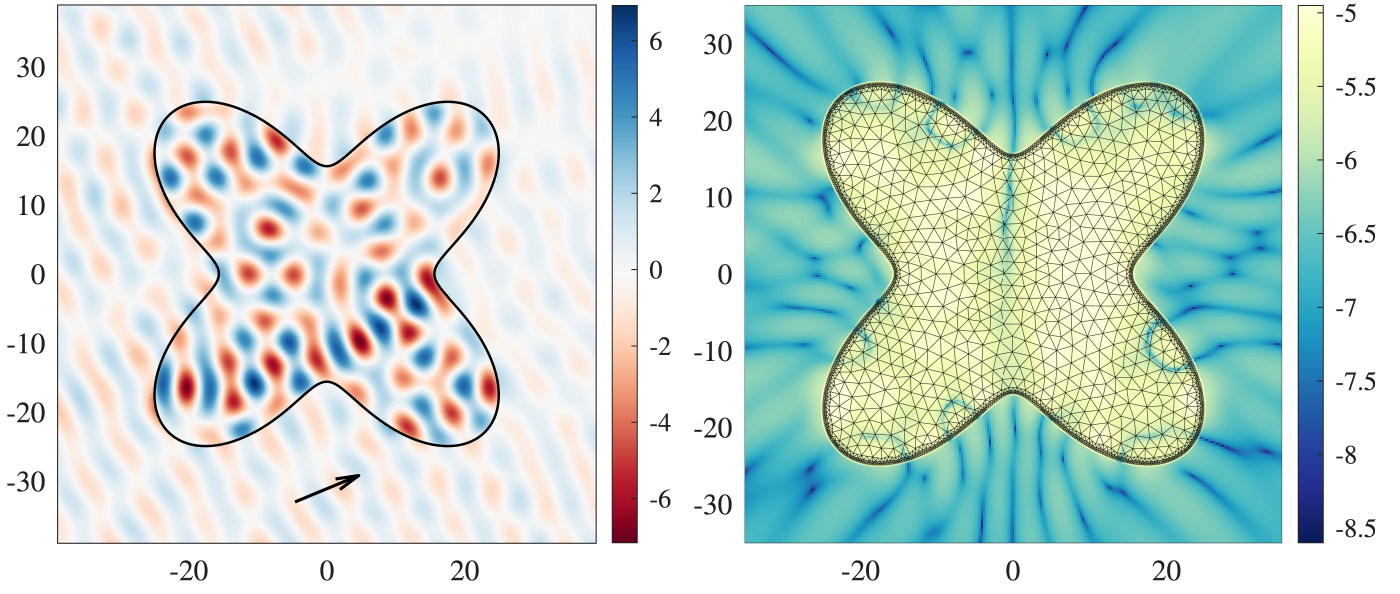}
		\caption{Exterior flexural-gravity waves with the free plate boundary conditions. The real part of $\partial_z \phi$ is plotted on the left, while the $\log_{10}$ relative error of self-convergence is plotted on the right. }
		\label{fig:flexerrorstar}
	\end{figure}

To demonstrate the qualitative differences in the surface wave problems and their associated exterior PDEs, the capillary-gravity, flexural-gravity, exterior Helmholtz, and exterior flexural wave problems were solved on an amorphous ``blob'' geometry, representing a surface contaminant (\Cref{fig:spikey_blob_plot}). The coefficient for the surface wave problems $\gamma = 0.25$ was chosen to be the same for both capillary and flexural problems, while the coefficients $\alpha = 0.38$ and $\beta = 0.49$ were chosen so that the positive real root $\rho_1$ was the same for both problems. The wavenumber $k$ for the exterior PDE problems was chosen to be the same $\rho_1$.

	\begin{figure}
		\centering
		\includegraphics[width=0.8\linewidth]{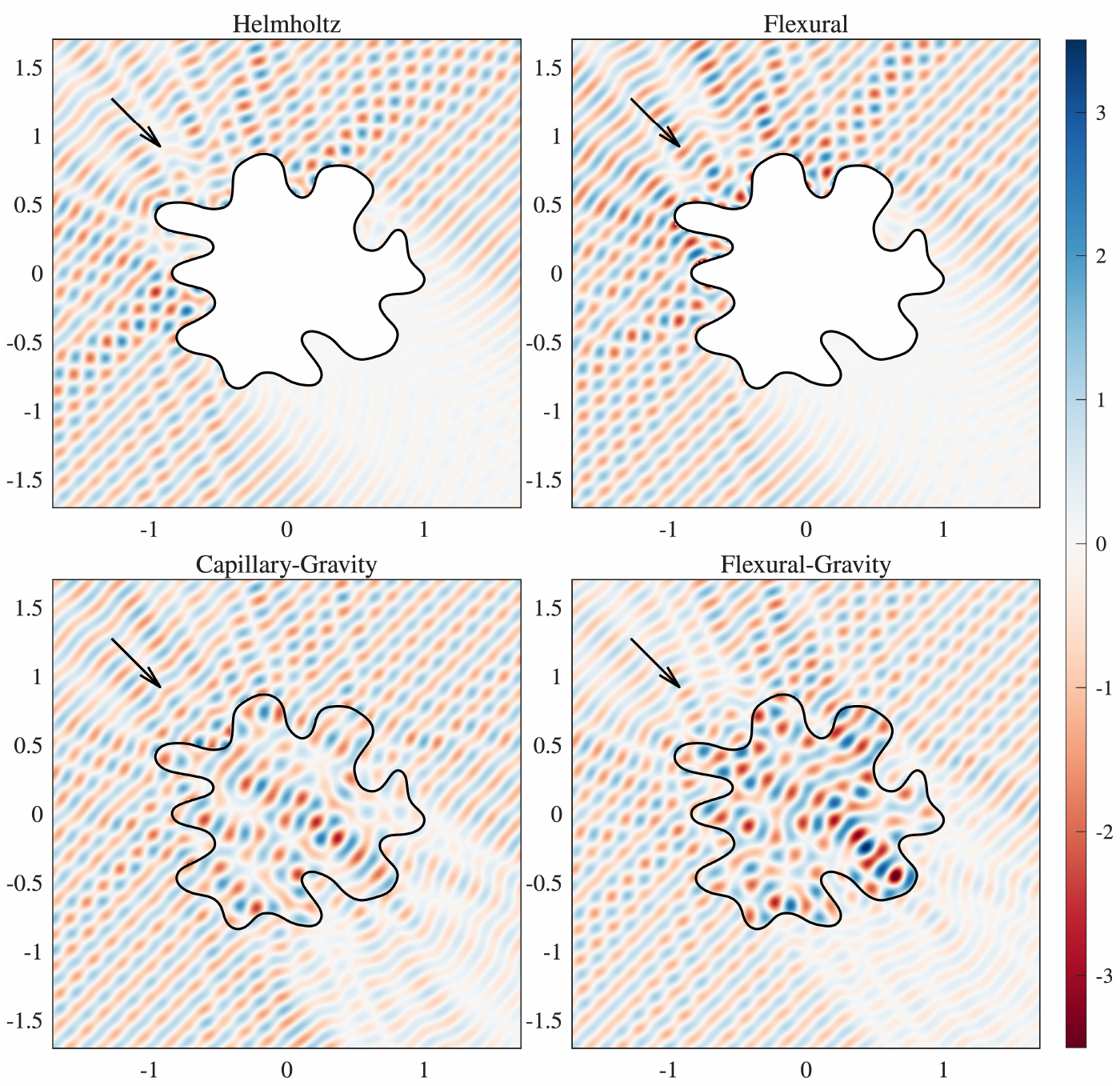}
		\caption{Plane wave scattering by an amorphous surface contaminant. Neumann BCs were imposed for the exterior Helmholtz problem and free plate BCs were imposed for the exterior flexural wave problem. The coefficients were chosen such that the wavenumber is $k = 38.1$ in the exterior and $k = 33.5$ in the interior (when it exists).}
		\label{fig:spikey_blob_plot}
	\end{figure}

    Lastly, the methods presented here are applied to modeling wave propagation around one of the largest ice `rifts', the WR2 rift on the Ross Ice Shelf in Antarctica. This rift belongs to a broader rift system that formed through tensile stress from ice flow \cite{ledoux2017structural}. The Ross Ice Shelf  is of particular interest to glaciologists because the deformation of its rifts gives rise to large iceberg calving events \cite{joughin2005calving}. Moreover, many studies have indicated that the growth of these rifts is triggered by tsunamis and other ocean wave forcing \cite{macayeal2006transoceanic,sergienko2010elastic,walker2013structural}. Icequakes and other seismic activity have been observed in the vicinity of these rifts during times of increased sea swell, leading to conjecture that flexural-gravity waves play an important role in their evolution \cite{bromirski2010transoceanic,chen2019ross}. 
    
    We model the intact ice shelf outside of the rift as a flexural-gravity wave medium while the rift interior, filled mostly with ice mélange and small bits of ice, is modeled using ordinary surface-gravity waves. For standard values of ice thickness (350 m) and incoming frequency (0.19 Hz), the dimensionless parameters for the flexural-gravity problem are given by $\alpha = 638$, $\gamma = -0.14$, and $\nu = 0.33$.  We take the incident field to be an incoming plane wave coming from the ocean. The total potential is shown in~\Cref{fig:ross_rift}. 
    The wave tends to be redirected along directions perpendicular to the boundary of the rift, and for this particular frequency, wave energy is localized within the western seam of the rift.

        \begin{figure}
            \centering
            \includegraphics[width=\linewidth]{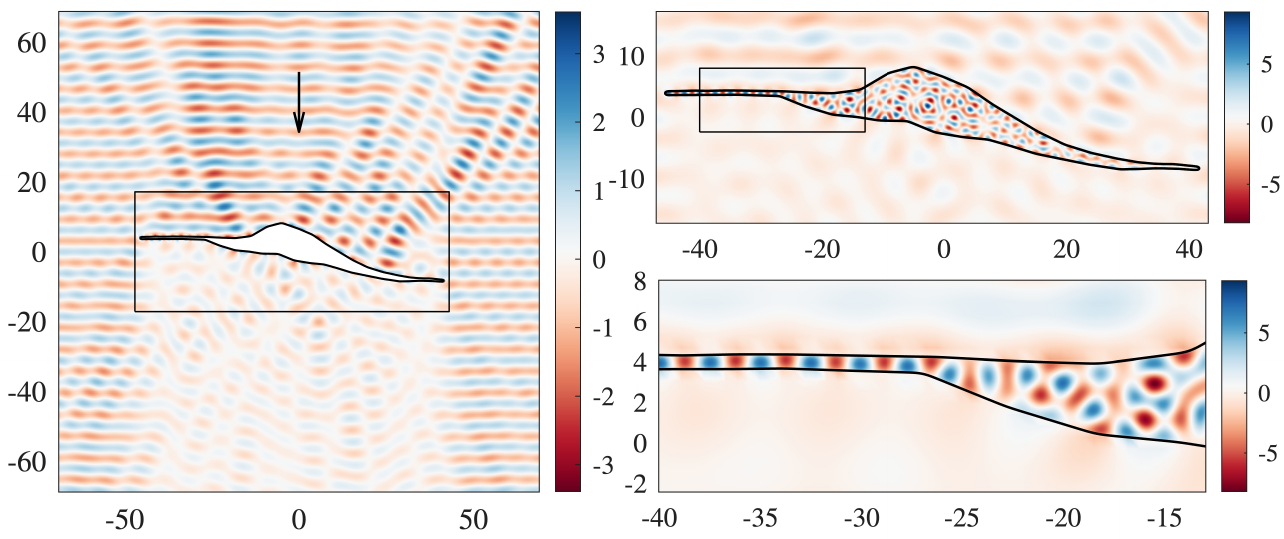}
            \vspace{-0.5cm}\caption{Flexural-gravity wave scattering from a rift in the Ross Ice Shelf. The left panel shows the field $\phi$ outside of the rift (arrow shows direction of the incident field from the ocean), while the top right panel shows the field inside the rift. The bottom right figure shows wave localization within one corner of the rift. The axis units are kilometers.}
            \label{fig:ross_rift}
        \end{figure}

	\section{Discussion and future work}
	\label{sec:discussion}
	
	In this paper we present a general framework for solving linear surface wave problems where the order of the derivatives in the surface-boundary condition experiences a jump between the interior and exterior regions. By representing the velocity potential as a single layer, this class of problems reduces to integro-differential equations on the unbounded surface of the fluid. Using the Green's function of the integro-differential operator, we derive second-kind integral equations with densities supported only in the interior region and its boundary. 
    
    We illustrate the application of this framework to problems involving capillary-gravity and flexural-gravity waves. Under certain natural assumptions, we prove that the resulting second-kind Fredholm equations are invertible. We then present a flexible and fast method for their numerical solution based on the precorrected FFT method. The scalability of our approach was demonstrated through several representative numerical examples, including a `rift' geometry based on the WR2 rift in the Ross Ice Shelf with realistic choices of physical parameters.
    
    While the methods discussed in this work are applied to models in which the exterior boundary condition has more derivatives than the interior boundary condition, similar techniques also apply when the interior region has more derivatives than the exterior region. This extension is being vigorously pursued. Moreover, the numerical methods presented in this work apply to a broader class of integro-differential equations, and nonlocal problems. While we use standard tools and extra refinement to treat the singularity present in $\mu$ near the boundary, it is relatively straightforward to develop more efficient tools based on tailored bases. A more specialized discretization approach, as well as the technical details of the fitted discretizations, fast quadrature generation, and numerical analysis of the proxy-annuli approach for computing precorrected FFT quadrature corrections, will be reported in an upcoming manuscript.

	\section{Acknowledgements}
	
The authors would like to thank Douglas MacAyeal, Zydrunas Gimbutas, and Mary Silber for many helpful discussions. JGH was partially supported by a Sloan Research Fellowship. This work was supported by the donors of ACS Petroleum Research Fund under New Directions Grant 68292-ND9.

	\appendix 
	
	\frenchspacing
	
	\bibliographystyle{siamplain}
	\bibliography{refs}
	
\end{document}